\documentclass[11pt, reqno]{amsart}
\usepackage[margin=1in]{geometry}

\usepackage{amssymb, amsmath, amsthm, amsfonts, bm, comment, enumerate}
\usepackage{bbm,dsfont}
\usepackage{xcolor, mathabx}

\usepackage{hyperref}

\def\R{\mathbb{R}}

\def\N{\mathbb{N}}
\def\C{\mathbb{C}}
\def\F{\mathbb{Q}_q}
\def\Z{\mathbb{Z}}
\def\T{\mathbb{T}}

\def\Q{\mathbb{Q}}

\renewcommand{\O}{\mathbb{Z}_{q}}

\def\mc{\mathcal}

\newcommand{\Om}[0]{\Omega}
\newcommand{\ta}[0]{\theta}
\newcommand{\bs}[0]{\setminus}
\newcommand{\ld}[0]{\lambda}
\newcommand{\vep}[0]{\varepsilon}

\newcommand{\lsm}[0]{\lesssim}

\newcommand{\wh}[1]{\widehat{#1}}

\newcommand{\ov}[1]{\overline{#1}}

\newcommand{\st}[1]{\substack{#1}}

\newcommand{\nms}[1]{\| #1 \|}

\numberwithin{equation}{section}
\theoremstyle{plain}
\newtheorem{thm}{Theorem}[section]
\newtheorem{prop}[thm]{Proposition}

\newtheorem{lem}[thm]{Lemma}
\newtheorem{corollary}[thm]{Corollary}

\newtheorem*{conj*}{Conjecture}

\newtheorem*{openproblem*}{Open Problem}

\title{Improved discrete restriction for the parabola}
\author{Shaoming Guo, \ \ Zane Kun Li,\ \ Po-Lam Yung}

\address{Shaoming Guo: Department of Mathematics, University of Wisconsin-Madison, Madison, WI-53706, USA}
\email{shaomingguo@math.wisc.edu}

\address{Zane Kun Li: Department of Mathematics, Indiana University Bloomington, Bloomington, IN-47405, USA}
\email{zkli@iu.edu}

\address{Po-Lam Yung: Department of Mathematics, The Chinese University of Hong Kong, Shatin, Hong Kong \quad \textit{and} \quad Mathematical Sciences Institute, The Australian National University, Canberra, Australia}
\email{plyung@math.cuhk.edu.hk \quad \textit{and} \quad polam.yung@anu.edu.au}

\begin{document}
\begin{abstract}
Using ideas from \cite{GMW} and working over $\Q_p$, we show that the discrete restriction
constant for the parabola is $O_{\vep}((\log M)^{2 + \vep})$.
\end{abstract}

\maketitle

\section{Introduction}
Let $e(z) := e^{2\pi i z}$ and let $K(M)$ denote the best constant such that
\begin{align}\label{dres}
\nms{\sum_{n = 1}^{M}a_{n}e(nx_1 + n^{2}x_2)}_{L^{6}([0, 1]^2)} \leq K(M)(\sum_{n = 1}^{M}|a_n|^{2})^{1/2}
\end{align}
for all sequences of complex numbers $\{a_n\}_{n = 1}^{M}$. Trivially,
$K(M) \leq M^{1/2}$.
In 1993, Bourgain in \cite{Bourgain93} considered, among other things, the size of $K(M)$ since \eqref{dres} is associated to the periodic Strichartz inequality for the nonlinear Schr\"{o}dinger equation on the torus. He obtained that
\begin{align}\label{eq:Bourgain_ul}
(\log M)^{1/6} \lsm K(M) \leq \exp(O(\frac{\log M}{\log\log M}))
\end{align}
using number theoretic methods, in particular the upper bound follows from the divisor bound and the lower bound follows
from Gauss sums on major arcs (see also \cite{BlomerBrudern} for a precise asymptotic in the case of $a_n = 1$ of \eqref{dres}).
It is natural to ask what is the true size of $K(M)$ and whether the gap between the upper
and lower bounds can be closed.

The lower bound has not been improved since \cite{Bourgain93}. However
by improving the upper bound on the decoupling constant for the parabola,
Guth-Maldague-Wang recently in \cite{GMW} improved the upper bound in \eqref{eq:Bourgain_ul}
to $\lsm (\log M)^{C}$ for some unspecified but large absolute constant $C$.
Our main result is that $C$ can be reduced to $2+$. More precisely:

\begin{thm}\label{final}
For every $\vep > 0$, there exists a constant $C_{\vep} > 0$ such that
$$K(M) \leq C_{\vep} (\log M)^{2 + \vep}.$$
\end{thm}

Our proof of Theorem \ref{final} will rely on a decoupling theorem for the parabola in $\Q_{p}$.
Previous work on studying discrete restriction using decoupling relied on proving decoupling theorems
over $\R$ (see for example \cite{BD, BDG, GMW, Li18}). Here, we will broadly follow the proof in \cite{GMW}
except to efficiently keep track of the number of logs we will prove a decoupling theorem over $\Q_p$ rather than over $\R$.
Additionally we will introduce some extra efficiencies to their argument to decrease the number of logs even further.

Working in $\Q_p$ has two benefits. First, the Fourier transform
of a compactly supported function is also compactly supported
and hence this allows us to rigorously and efficiently apply the uncertainty principle which is just a heuristic in $\R$.
Second, since 6 is even, decoupling in $\Q_p$ still implies discrete restriction estimates.

To avoid confusing the $p$ in $\Q_{p}$ with the $p$ in $L^{p}$ norm, henceforth we will replace the $p$ in $\Q_{p}$ with $q$.

Let $q$ be a fixed odd prime. Let $|\cdot |$ be the $q$-adic norm associated to $\F$.
We omit the dependence of this norm on $q$.
This is a slight abuse of notation as we will use the same notation for the absolute value on $\C$, as well as the length of a $q$-adic interval. However, the meaning of the symbol will be clear from context.
In Section \ref{basic}, we summarize all relevant facts of $\F$ that we make use of.
See Chapters 1 and 2 of \cite{Taibleson} and Chapter 1 (in particular Sections 1 and 4) of \cite{VVZ} for a more complete discussion of analysis on $\F$.

For $\delta \in q^{-\N}$, we write
\begin{equation*}
\Xi_{\delta} = \{(\xi,\eta) \in \F^2 \colon \xi \in \O, |\eta - \xi^2| \leq \delta\}.
\end{equation*}
For a Schwartz function $F: \F^{2} \rightarrow \C$ and an interval $\tau \subset \O$, let $F_{\tau}$ be defined by
$\wh{F_{\tau}} := \wh{F} \, 1_{\tau \times \F}$.
Our main decoupling theorem is as follows and is the $\F$ analogue of Theorem 1.2 of \cite{GMW}.
\begin{thm} \label{thm:L2Linfty_intro}
For every odd prime $q$ and every $\varepsilon > 0$, there exists a constant $C_{\varepsilon,q}$, such that whenever $R \in q^{2\N}$ and a Schwartz function $F \colon \F^2 \to \C$ has Fourier support contained in $\Xi_{1/R}$, one has
\begin{equation}\label{L2Linfty_eq}
\int_{\F^2} |F|^6 \leq C_{\varepsilon,q} (\log R)^{12+\varepsilon} (\sum_{|\tau| = R^{-1/2}} \|F_{\tau}\|_{L^{\infty}(\F^2)}^2)^2 (\sum_{|\tau| = R^{-1/2}} \|F_{\tau}\|_{L^2(\F^2)}^2).
\end{equation}
Here the sums on the right hand side are over all intervals $\tau \subset \O$ with length $R^{-1/2}$.
\end{thm}
This theorem is proved in Sections \ref{log12:setup}-\ref{highlow}.
We will in fact show this theorem with $\varepsilon$ replaced by $10\varepsilon$.
Since 6 is even, Theorem \ref{thm:L2Linfty_intro} once again immediately implies Theorem \ref{final} (as we prove in Section \ref{decimpliesres}).

The 12 powers of log in \eqref{L2Linfty_eq} can be accounted for as follows.
Reducing from \eqref{L2Linfty_eq} to the level set estimate (Proposition \ref{prop:newprop34}) costs 5 logs. They come from: 3 logs from the Whitney decomposition in Section \ref{whitney}, 1 log from the number of scales in deriving \eqref{eq:broad_final}, and 1 log from pigeonholing to derive \eqref{3.6}.
The level set estimate itself costs 7 logs. These come from:
1 log since we decompose $\F^2$ into sets $\Om_k$ and $L$ in Section \ref{7.3decomp} and \eqref{eq:goal2},
2 logs to control
$g^{2}_k$ by $|g_{k}^{h}|^{2}$ on $\Om_k$ in \eqref{eq:7.55}, and 4 logs from the appearance of $\lambda^{2}$
in \eqref{eq:lambda2appear} (also see \eqref{eq:lambda}).

In addition to efficiencies introduced by working with the uncertainty principle $q$-adically, we introduce a Whitney decomposition,
much like in \cite{GLYZK}, which allows us to more efficiently reduce to a bilinear decoupling problem.
Additionally compared to \cite{GMW}, the ratio between our successive scales $R_{k + 1}/R_{k}$ is of size $O((\log R)^{\vep})$ rather than
in $O((\log R)^{12})$ which allows for further reductions (we essentially have $O(\vep^{-1})$ times many more scales than in \cite{GMW}).
Note that \eqref{L2Linfty_eq} is not a true $\F$ analogue of a
$l^2 L^6$ decoupling theorem for the parabola. At the cost of a
few more logs, a similar argument as in Section 5 of \cite{GMW} would allow us
to upgrade to an actual
$l^2 L^6$ decoupling theorem, however \eqref{L2Linfty_eq} is already
enough for discrete restriction for the parabola.

Since $p$-adic intervals correspond to residue classes it may be possible to rewrite the proof of Theorem \ref{thm:L2Linfty_intro} in the language of congruences and compare it with efficient congruencing \cite{WooleyNested}.
However we do not attempt this here. For more connections between efficient congruencing and decoupling see \cite{GLY, GLYZK, Li18, Pierce}.

In this paper we consider decoupling over $\Q_p$. However one can also consider
the restriction and Kakeya conjectures over $\Q_p$ (or alternatively over
more general local fields). We refer the interested
reader to \cite{HickmanWright} and the references therein for more discussion.

For the rest of the paper, for two positive expressions $X$ and $Y$, we write $X \lsm Y$ if $X \leq C_{\vep, q} Y$ for some constant $C_{\vep, q}$ which is allowed
to depend on $\vep$ and $q$. We write $X \sim Y$ if $X \lsm Y$ and $Y \lsm X$.
Additionally by writing $f(x) = O(g(x))$, we mean $|f(x)| \lsm g(x)$. Finally, we say that $f$ has Fourier support in $\Om$ if its Fourier transform $\wh{f}$ is supported in $\Om$.

\subsection*{Acknowledgements}
SG is supported by NSF grant DMS-1800274. ZL is supported by NSF grant DMS-1902763. PLY is partially supported by a Future Fellowship FT200100399 from the Australian Research Council.
\section{Some basic properties of $\F$}\label{basic}
For convenience we briefly summarize some key relevant facts about $\F$.
First, for a prime $q$, $\F$ is the completion of the field $\Q$ under the $q$-adic norm, defined by $|0| = 0$ and $|q^a b/c| = q^{-a}$ if $a \in \Z$, $b, c \in \Z \setminus \{0\}$ and $q$ is relatively prime to both $b$ and $c$. Then $\F$ can be identified (bijectively) with the set of all formal series
\[
\F = \left\{\sum_{j=k}^{\infty} a_j q^j \colon k \in \Z, a_j \in \{0, 1, \dots, q-1\} \text{ for every $j \geq k$} \right\},
\]
and the $q$-adic norm on $\F$ satisfies
$|\sum_{j=k}^{\infty} a_j q^j| = q^{-k}$ if $a_k \ne 0$.

The $q$-adic norm obeys the ultrametric inequality $|x + y| \leq \max\{|x|, |y|\}$ with equality when $|x| \neq |y|$. We also define the $q$-adic norm on $\F^2$ by setting $|(x,y)| = \max\{|x|,|y|\}$ for $(x,y) \in \F^2$.

Write $\O = \{x \in \F: |x| \leq 1\}$ for the ring of integers of $\F$. This is in analogy to the real interval $[-1, 1]$.
In analogy to working over $\R$, for $a \in \O$, we will call sets of the form $\{\xi \in \O: |\xi - a| \leq q^{-b}\}$
an interval inside $\O$ of length $q^{-b}$ (so the length of an interval coincides with its diameter, i.e. maximum distance between two points in that interval). Similarly for $(c_{1}, c_{2}) \in \F^{2}$, we will call sets of the form
$\{(x, y) \in \F^{2}: |x - c_1| \leq q^{-b}, |y - c_2| \leq q^{-b}\}$ a square of side length $q^{-b}$.
Note that because the norm on $\F^{2}$ is the maximum $q$-adic norm of each coordinate, this square is the same
as $\{(x, y) \in \F^{2}: |(x, y) - (c_{1}, c_{2})| \leq q^{-b}\}$. Thanks to the ultrametric inequality, if two squares intersect, then one is contained inside the other; hence two squares of the same size are either equal or disjoint.

Observe that $\O$ is a subset of $\F$ consisting of elements of the form $\sum_{j \geq 0}a_{j}q^{j}$ where $a_{j} \in \{0, 1, \ldots, q - 1\}$.
Since each positive integer has a base $q$ representation, we may embed $\N$ into $\O$. Identifying $-1$ with the element
$\sum_{j \geq 0}(q - 1)q^{j}$ in $\O$ then allows us to embed $\Z$ into $\O$.

Note that if $\ell \in \N$, the intervals $\{\xi \in \O: |\xi - a| \leq 1/q^{\ell}\}$ for $a = 0, 1, \ldots, q^{\ell} - 1$ partition
$\O$ into $q^{\ell}$ many disjoint intervals which are pairwise disjoint and each pair of intervals are separated by distance at least $q^{-\ell + 1}$.
To see this, suppose $|\xi_1 - a| \leq q^{-\ell}$ and $|\xi_{2} - b| \leq q^{-\ell}$ for some $a \neq b$.
As $|a - b| \geq q^{-\ell + 1}$ and $|(\xi_1 - \xi_2) - (a - b)| \leq q^{-\ell}$, the equality case of
the ultrametric inequality implies that $|\xi_1 - \xi_2| = |a - b| \geq q^{-\ell + 1}$.

Next, for fixed $a \in \{0, 1, \ldots, q^{\ell} - 1\}$, the interval $\{\xi \in \O: |\xi - a| \leq 1/q^{\ell}\}$ is exactly the $\xi \in \O$
such that $\xi \equiv a \pmod{q^{\ell}}$ (meaning $q^{-\ell} (\xi - a) \in \O$). This illustrates the connection between $q$-adic intervals in $\F$ and residue classes and both point of views are useful throughout; for instance, it follows easily now that $\O$ is the union of these $q^{\ell}$ disjoint intervals.

Finally, let $\chi$ be the additive character of $\F$ that is equal to $1$ on $\O$ and non-trivial on $q^{-1} \O$ (up to isomorphism, there is essentially just one, given by
\[
\chi(x) := e\Big( \sum_{j=k}^{-1} a_j q^j \Big) \quad \text{if $x = \sum_{j=k}^{\infty} a_j q^j$}
\]
where $a_j \in \{0,\dots,q-1\}$ for all $j$).
From this, one can define the Fourier transform for $f \in L^{1}(\F)$ by $\wh{f}(\xi) := \int_{\F}f(x)\chi(-\xi x)\, dx$ for $\xi \in \F$, where $dx$ is the Haar measure on $\F$,
and we have an analogous definition for the Fourier transform in higher dimensions.
The theory of the Fourier transform
in $\F$ is essentially the same as in $\R$ and we refer the interested reader to \cite{Taibleson, VVZ} for more details.
Note that in $\F$ and in higher dimensions, linear combinations of indicator functions of intervals and squares play the analogue of Schwartz functions in the real setting. For $f, g \in L^1(\F^2) \cap L^2(\F^2)$, we have Plancherel's identity
$
\int_{\F^2} f \, \overline{g} = \int_{\F^2} \hat{f} \, \overline{\hat{g}},
$
which allows one to extend the Fourier transform to a unitary operator on $L^2(\F^2)$. We also have $\widehat{f*g} = \hat{f} \hat{g}$ for any integrable $f$ and $g$ on $\F^2$, where $(f*g)(x)$ is the convolution $\int_{\F^2} f(x-y) g(y) dy$. The inverse Fourier transform will be denoted by $\widecheck{\cdot}$, and we have $f = \widecheck{\hat{f}}$ for Schwartz functions $f$. Henceforth we will only deal with Schwartz functions on $\F^2$; note $F_{\tau}$ is Schwartz whenever $F$ is Schwartz.

\subsection{Basic geometry and the uncertainty principle}
The key property about harmonic analysis in $\F$ is that the Fourier transform of an indicator function of an interval is another
indicator function of an interval.
The key lemma is following, for a proof see p.42 of \cite{VVZ}.
\begin{lem}\label{lem:char}
For $\xi \in \F$ and $\gamma \in \Z$,
\begin{align*}
\widecheck{1_{|x| \leq q^{\gamma}}}(\xi) = \int_{|x| \leq q^{\gamma}}\chi(\xi x)\, dx = q^{\gamma}(1_{|\xi| \leq q^{-\gamma}})(\xi).
\end{align*}
\end{lem}

Another useful geometric fact about $\F^2$ is that curvature disappears entirely if one considers the intersection of $\Xi_{1/R}$ with a vertical strip of width $R^{-1/2}$.

\begin{lem}
For any $R \in q^{2\Z}$ and any interval $I \subset \F$ with length $|I| = R^{-1/2}$, the set $\{(\xi,\eta) \in \F^2 \colon \xi \in I, |\eta - \xi^2| \leq R^{-1}\}$ coincides with the parallelogram
\[
\{(\xi,\eta) \in \F^2 \colon |\xi - a| \leq R^{-1/2}, |\eta - 2 a \xi + a^2| \leq R^{-1}\}
\]
where $a$ is any point in $I$.
\end{lem}

\begin{proof}
Let $a \in I$. The ultrametric inequality implies $I = \{\xi \in \F \colon |\xi - a| \leq R^{-1/2}\}$. Now $|\eta - \xi^2| = |\eta - a^2 - 2a(\xi-a) - (\xi-a)^2| = |(\eta - 2 a \xi + a^2) - (\xi-a)^2|$. It follows that for $\xi \in I$, i.e. if $|\xi - a| \leq R^{-1/2}$, then $|\eta - \xi^2| \leq R^{-1}$, if and only if $|\eta - 2 a \xi + a^2| \leq R^{-1}$.
\end{proof}

This motivates the following rigorous $q$-adic uncertainty prinicple, that is just a heuristic in $\R$.
\begin{lem}[Uncertainty Principle]\label{lem:uncertainty}
Let $R \in q^{2\Z}$ and $I \subset \F$ be an interval of length $|I| = R^{-1/2}$.
Define the parallelogram
\begin{align} \label{eq:parallelogram_def}
P := \{(\xi, \eta) \in \F^{2}: \xi \in I, |\eta - \xi^2| \leq R^{-1}\}
\end{align}
and the dual parallelogram
\begin{align} \label{eq:ta^ast}
T := \{(x, y) \in \F^{2}: |x + 2ay| \leq R^{1/2}, |y| \leq R\}
\end{align}
where $a$ is any point in $I$ (this is well-defined independent of the choice of $a$).
Let $f$ be Schwartz and Fourier supported in $P$. Then $|f|$ is constant on each translate of $T$.
\end{lem}
\begin{proof}
One only needs to prove this for $I = \O$, $R = 1$ and then invoke affine invariance. Alternatively, and more directly, we have
\begin{align*}
\widecheck{1_{P}}(x, y) &= \int_{|t| \leq R^{-1}}\int_{|s - a| \leq R^{-1/2}}\chi(sx + s^{2}y)\chi(ty)\, ds\, dt\\
&= \chi(ax + a^{2}y)(\int_{|s| \leq R^{-1/2}}\chi(s(x + 2ay) + s^{2}y)\, ds)R^{-1}1_{|y| \leq R}
\end{align*}
where the last equality is by Lemma \ref{lem:char}.
Since $|y| \leq R$, $|s^{2}y| \leq 1$ and therefore $s^2 y \in \O$. As $\chi$ is trivial on $\O$, after another application of Lemma \ref{lem:char}, the above expression is equal to
$R^{-3/2}\chi(ax + a^{2}y)1_{|x + 2ay| \leq R^{1/2}, |y| \leq R} = R^{-3/2}\chi(ax + a^{2}y)1_{T}$.

Suppose $(x, y) \in (A, B) + T$ for some $(A, B) \in \F^2$.
Write $x = A + x'$ and $y = B + y'$ for some $(x', y') \in T$.
Then since $f = f \ast \widecheck{1_{P}}$, we have
\begin{align}\label{eq:unceq1}
f(x, y) = R^{-3/2}\chi(ax + a^{2}y)\int_{\F^{2}} f(z, w)\chi(-az - a^{2}w)1_{T}(x' + A - z, y' + B - w)\, dz\, dw
\end{align}
Since $|x' + 2ay'| \leq R^{1/2}$, using the ultrametric inequality, $|(x' + A - z) + 2a(y' + B - w)| \leq R^{1/2}$
if and only if $|(A - z) + 2a(B - w)| \leq R^{1/2}$. Similarly, since $|y'| \leq R$, $|y' + B - w| \leq R$
if and only if $|B - w| \leq R$. Therefore \eqref{eq:unceq1} is equal to
\begin{align*}
R^{-3/2}\chi(ax + a^{2}y)\int_{\F^2}f(z, w)\chi(-az - a^{2}w)1_{T}(A - z, B - w)\, dz\, dw.
\end{align*}
Thus $|f(x, y)|$ is independent of $(x, y) \in (A,B) + T$ and therefore $|f|$ is constant on each translate of $T$ (with a constant
that depends on $f$, $P$, $I$, and the particular translate of $T$).
\end{proof}
A similar proof as above shows that if $f$ is Fourier supported in a square of side length
$L$, then $|f|$ is constant on any square of side length $L^{-1}$. Furthermore, if $f$ is Fourier supported in a square centered at the origin of side length $L$, then $f$ itself is constant on any square of side length $L^{-1}$.

In analogy with the real setting, we will say that the parallelogram $T$ in \eqref{eq:ta^ast} has direction $(-2a, 1)$.
These parallelograms $T$ enjoy the following nice geometric properties.

\begin{lem} \label{lem:geom}
If $R \in q^{2\N}$, $I \subset \O$ is an interval with $|I| = R^{-1/2}$, and $T$ is the parallelogram defined by \eqref{eq:ta^ast} (with $a \in I$), then
\begin{enumerate}[(a)]
    \item \label{lem2.4a} each translate of $T$ is the union of $R^{1/2}$ many squares of side length $R^{1/2}$;
    \item any two translates of $T$ are either equal or disjoint;
    \item \label{lem2.4c} %
    any square of side length $R$ can be partitioned into translates of $T$.
\end{enumerate}
We write $\T(I)$ for the set of all translates of $T$. Note that $(c)$ implies that
$\F^2$ can be tiled by translates of $T$.
\end{lem}

\begin{proof}
\begin{enumerate}[(a)]
    \item First, we claim that if $(x,y) \in T$, and $|(x',y') - (x,y)| \leq R^{1/2}$, then $(x',y') \in T$ as well. This is because $|x'+2ay'| = |x+2ay+(x'-x)+2a(y'-y)| \leq R^{1/2}$ if both $|x+2ay| \leq R^{1/2}$ and $|(x',y')-(x,y)| \leq R^{1/2}$ (recall $|2a| \leq 1$ when $a \in \O$). Similarly, $|y| \leq R$ and $|y'-y| \leq R^{1/2}$ implies $|y'| \leq R$. This proves the claim. It follows that if $(x,y)$ belongs to a certain translate of $T$, then the square of side length $R^{1/2}$ containing $(x,y)$ is also contained in the same translate of $T$.

    Now by the ultrametric inequality, two squares of side length $R^{1/2}$ are either equal or disjoint. Thus every translate of $T$ is a union of squares of side lengths $R^{1/2}$, and volume considerations show that each translate of $T$ contains $R^{1/2}$ many such squares.
    \item It suffices to show that if $(x,y) + T$ intersects $T$, then $(x,y) \in T$ (because then $(x,y)+T = T$). But if $(x,y) + T$ and $T$ both contains a point $(x',y')$, then both $|(x'-x)+2a(y'-y)| \leq R^{1/2}$ and $|x'+2ay'| \leq R^{1/2}$, which implies $|x+2ay| \leq R^{1/2}$. Similarly, $|y'-y| \leq R$ and $|y'| \leq R$ implies $|y| \leq R$. Thus $(x,y) \in T$, as desired.
    \item Write $R = q^{2A}$ for $A \geq 1$.
It suffices to partition $Q = \{(x, y) \in \F^2: |x| \leq R, |y| \leq R\}$ into translates of parallelograms $T_{a} := \{(x, y) \in \F^2: |x+ 2ay| \leq R^{1/2}, |y| \leq R\}$.

We first consider the $a = 0$ case.
Let $S = \{\sum_{-2A \leq j < -A}a_{j}q^{j}: a_{j} \in \{0, 1, \ldots, q - 1\}\}$.
Note that $\# S = R^{1/2}$.

We claim we can tile $Q$ by $\{(s, 0) + T_{0}: s \in S\}$. Indeed, for each $(x, y) \in Q$, we can write
$x = \sum_{-2A \leq j < -A}x_{j}q^{j} + \sum_{j \geq -A}x_{j}q^{j}$ for some $x_j \in \{0, 1, \ldots, q - 1\}$.
As $\sum_{-2A \leq j < -A}x_{j}q^{j} \in S$, $x \in (\sum_{-2A \leq j < -A}x_{j}q^{j}, 0) + T_0$.
This shows $Q \subset \bigcup_{s \in S}(s, 0) + T_0$.
The ultrametric inequality implies that $(s, 0) + T_0 \subset Q$ for each $s \in S$ and so $Q = \bigcup_{s \in S}(s, 0) + T_0$.

Finally, this union is disjoint as if $(x, y) \in ((s_1, 0) + T_0) \cap ((s_2, 0) + T_0)$, then $|s_1 - s_2| \leq R^{1/2}$
but from the definition of $S$, $|s_1 - s_2| \geq q^{A + 1} = R^{1/2} q$.
Therefore we have partitioned $Q$ into translates of $T_0$.

Next we consider the general case. Let $L_{a} = (\begin{smallmatrix} 1 & 2a \\0 & 1\end{smallmatrix})$.
The ultrametric inequality gives that $L_{a}(Q) = Q$ since $|2a| \leq 1$ and for $s \in S$, $L_{a}((s, 0) + T_0) = (s, 0) + T_a$.
Therefore we can also partition $Q$ into translates of $T_a$.
    \end{enumerate}
\end{proof}

\begin{corollary} \label{cor:wavepacket} %
Let $R \in q^{2\N}$, $I \subset \O$ be an interval with $|I| = R^{-1/2}$, and $f$ be a Schwartz function with Fourier support in $\{(\xi,\eta) \in \F^2 \colon \xi \in I, |\eta - \xi^2| \leq 1/R\}$. Then there exist constants $\{c_T\}_{T \in \T(I)}$ such that
\begin{equation} \label{eq:ptwisewp}
|f| = \sum_{T \in \T(I)} c_T 1_T.
\end{equation}
As a result, $|f|^2 = \sum_{T \in \T(I)} c_T^2 1_T$, and
\[
\int_{\F^2} |f|^2 = \sum_{T \in \T(I)} c_T^2 |T|.
\]
\end{corollary}

\begin{proof}
By Lemma~\ref{lem:uncertainty}, for every $T \in \T(I)$, there exists a constant $c_T$ so that $|f| = c_T$ on $T$. By Lemma~\ref{lem:geom}(\ref{lem2.4c}), $\T(I)$ tiles $\F^2$. Thus \eqref{eq:ptwisewp} holds and the rest follows easily.
\end{proof}

\begin{lem}\label{lem:tubeintersect}
Suppose $R \in q^{2\N}$ and $a, b \in \O$ with $a \neq b$, let
$$T = \{(x, y) \in \F^2: |x + 2ay| \leq R, |y| \leq R^2\}$$
and
$$T' = \{(x, y) \in \F^2: |x + 2by| \leq R, |y| \leq R^2\}.$$
Then
\begin{align*}
|T \cap T'| \leq \frac{R^2}{|b - a|}.
\end{align*}
\end{lem}
\begin{proof}
By redefining $x$, we may assume that $a = 0$. Then
\begin{align*}
T \cap T' &= \{(x, y) \in \F^2: \max(|x|, |x + 2by|) \leq R, |y| \leq R^2\}\\
&\subset \{(x, y) \in \F^2: |x| \leq R, |y| \leq R/|2b|\}.
\end{align*}
Since $q$ is an odd prime, the claim then follows since
the Haar measure is normalized so that $|\O| = 1$.
\end{proof}

\section{Theorem \ref{thm:L2Linfty_intro} implies Theorem \ref{final}}\label{decimpliesres}
Since $K(M)$ is trivially increasing, it suffices to show Theorem \ref{final}
only in the case when $M = q^{t}$ for some $t \in \N$. By using the trivial bound for $K(M)$,
we may also assume that $t$ is sufficiently large (depending only on an absolute constant).
By considering real and imaginary parts,
we may also assume that $a_n$ is a sequence of real numbers in \eqref{dres}.

Let $R = M^2 = q^{2t}$.
Choose $F$ such that
\begin{align*}
\wh{F}(\xi, \eta) = \sum_{n = 1}^{q^{t}}a_{n}1_{(n, n^2) + B(0, q^{-10t})}(\xi, \eta)q^{20t}.
\end{align*}
Here we are using the embedding of $\Z$ into $\O$, and $(n,n^2)+B(0,q^{-10t})$ denotes the square $\{(\xi,\eta) \in \F^2 \colon |(\xi,\eta)-(n,n^2)| \leq q^{-10t}\}$.
Note that $\wh{F}$ is indeed supported inside $\Xi_{1/R}$ since %
if $|(\xi, \eta) - (n, n^2)| \leq q^{-10t}$ for some $n \in \N$, then $\xi \in \O$ and %
\begin{align*}
|\xi^{2} - \eta| &= |(\xi - n)^{2} + 2n(\xi - n) + n^{2} - \eta|\leq \max(|\xi - n|^{2}, |2n||\xi - n|, |n^{2} - \eta|).
\end{align*}
Since $q \geq 3$ is an odd prime, $|2n| \leq 1$ and so the above is $\leq q^{-10t} \leq q^{-2t}$.

Inverting the Fourier transform gives that
\begin{align*}
F(x) = \bigg(\sum_{n = 1}^{q^{t}}a_{n}\chi(x_1 n + x_2 n^{2})\bigg)1_{B(0, q^{10t})}(x).
\end{align*}
Similarly, for each $\tau$ on the right hand side of \eqref{L2Linfty_eq} (with length $R^{-1/2} = M^{-1} = q^{-t}$),
$F_{\tau}(x) =a_{n}\chi(x_1 n + x_2 n^{2})1_{B(0, q^{10t})}(x)$
where $n$ is the unique element in $\{1, \dots, q^t\} \cap \tau$; then
$\nms{F_{\tau}}_{L^{\infty}(\F^2)}^{2} = |a_n|^{2}$
and
$\nms{F_{\tau}}_{L^{2}(\F^2)}^{2} = |a_{n}|^{2}q^{20t}$.
The right hand side of \eqref{L2Linfty_eq}
is then $\lesssim (\log M)^{12 + 10\vep}q^{20t}(\sum_{n = 1}^{q^{t}}|a_n|^2)^3$.

It now remains to show that
\begin{align}\label{targetdr}
\nms{F}_{L^{6}(\Q_{q}^{2})}^{6} = q^{20t}\nms{\sum_{n = 1}^{q^{t}}a_{n}e(nx_1 + n^{2}x_2)}_{L^{6}([0, 1]^2)}^{6}.
\end{align}
This relies on that we are working with $L^6$.
Expanding the left hand side gives
\begin{equation}
\begin{aligned}\label{decreseq1}
&\sum_{n_1, \ldots, n_6 = 1}^{q^{t}}a_{n_1}\cdots a_{n_6}\times\\
&\int_{B(0, q^{10t})}\chi((n_1 + n_2 + n_3 - n_4 - n_5 - n_6)x_1 + (n_{1}^{2} + n_{2}^{2} + n_{3}^{2} - n_{4}^{2} - n_{5}^{2} - n_{6}^{2})x_2)\, dx.
\end{aligned}
\end{equation}
Applying Lemma \ref{lem:char} gives that the above is equal to
\begin{align*}
\sum_{n_1, \ldots, n_6 = 1}^{q^{t}}q^{20t}a_{n_1} \cdots a_{n_6}1_{|n_1 + n_2 + n_3 - n_4 - n_5 - n_6| \leq q^{-10t}}1_{|n_{1}^{2} + n_{2}^{2} + n_{3}^{2} - n_{4}^{2} - n_{5}^{2} - n_{6}^{2}| \leq q^{-10t}}.
\end{align*}
The statement that $(n_1, \ldots, n_6) \in \{1, \ldots, q^{t}\}^{6}$ are such that
\begin{equation}\label{eq:padicineq}
|n_1 + n_2 + n_3 - n_4 - n_5 - n_6| \leq q^{-10t}, \quad |n_{1}^{2} + n_{2}^{2} + n_{3}^{2} - n_{4}^{2} - n_{5}^{2} - n_{6}^{2}| \leq q^{-10t}
\end{equation}
is equivalent to the statement that $(n_1, \ldots, n_6)\in \{ 1, \ldots, q^{t}\}^{6}$ are such that
\begin{align*}
n_{1} + n_2 + n_{3} - n_{4} - n_{5} - n_{6} \equiv 0 \pmod{q^{10t}}, \quad n_{1}^{2} + n_{2}^{2} + n_{3}^{2} - n_{4}^{2} - n_{5}^{2} - n_{6}^{2} \equiv 0 \pmod{q^{10t}}.
\end{align*}
Since the $1 \leq n_i \leq q^{t}$, $n_{1} + n_{2} + n_{3} - n_{4} - n_{5} - n_{6}$ is an integer between
$-3q^{t}$ and $3q^{t}$, while $n_{1}^{2} + n_{2}^{2} + n_{3}^{2} - n_{4}^{2} - n_{5}^{2} - n_{6}^{2}$ is an integer 
between $-3q^{2t}$ and $3q^{2t}$. Since the only integer $\equiv 0 \pmod{q^{10t}}$ between $-3q^{2t}$
and $3q^{2t}$ is 0, \eqref{eq:padicineq} is true for a given $(n_1, \ldots, n_6) \in \{1, \ldots, q^t\}^{6}$ if
and only if
\begin{align*}
n_{1} + n_2 + n_{3} - n_{4} - n_{5} - n_{6}  = 0, \quad n_{1}^{2} + n_{2}^{2} + n_{3}^{2} - n_{4}^{2} - n_{5}^{2} - n_{6}^{2} = 0.
\end{align*}
Thus \eqref{decreseq1} is equal to
\begin{align*}
q^{20t}\sum_{n_1, \ldots, n_6 = 1}^{q^{t}}a_{n_1} \cdots a_{n_6}1_{n_1 + n_2 + n_3 - n_4 - n_5 - n_6 = 0}1_{n_{1}^{2} + n_{2}^{2} + n_{3}^{2} - n_{4}^{2} - n_{5}^{2} - n_{6}^{2} = 0}
\end{align*}
which in turn is equal to the right hand side of \eqref{targetdr}.

\section{Setting up many scales for the proof of Theorem~\ref{thm:L2Linfty_intro}}\label{log12:setup}

We now set out to prove Theorem~\ref{thm:L2Linfty_intro}.
Fix $\varepsilon \in (0,1)$. Let $A$ be an integer with
\begin{equation*}
\frac{1}{\varepsilon} \leq A \leq \frac{2}{\varepsilon}.
\end{equation*}
Henceforth all implicit constants may depend on $q$, $\varepsilon$ and $A$.

Given $R \in q^{2\N}$, choose $r \in 4\N$ so that
\begin{equation*}
q^{q^{A(r-4)}} \leq R < q^{q^{Ar}}.
\end{equation*}
Then $q^{Ar} \sim \log R$ and $(\log R)^{\varepsilon/2} \lesssim q^r \lesssim (\log R)^{\varepsilon}$,
so for $R$ sufficiently large (depending only on $q$ and $\varepsilon$) we have $r \sim \log \log R$.
Henceforth we fix a sufficiently large $R$, and define
\begin{equation*}
R_k := q^{kr} \quad \text{for $k = 0, 1, \dots, N$},
\end{equation*}
where $N \in \N$ is defined such that
\begin{equation*}
q^{Nr} \leq R < q^{(N+1)r}.
\end{equation*}
The choice $r \in 4\N$ ensures that
\begin{equation} \label{eq:Rk1/4}
R_k^{-1/2} \in q^{-2\N}
\end{equation}
for every $k$. Throughout we write $\tau_k$ for a generic interval inside $\O$ of length $R_k^{-1/2}$, for $k = 0, 1, \dots, N$. For instance, $\sum_{\tau_N}$ means sums over all intervals $\tau_N \subset \O$ with $|\tau_N| = R_N^{-1/2}$.

Let $F \colon \F^2 \to \C$ be Fourier supported in $\Xi_{1/R}$ as in the statement of Theorem~\ref{thm:L2Linfty_intro}. In order to establish \eqref{L2Linfty_eq}, it suffices to prove
\begin{equation} \label{eq:goal0}
\int_{\F^2} |F|^6 \lesssim (\log R)^{12+ 9\varepsilon} (\sum_{\tau_N} \|F_{\tau_N}\|_{L^{\infty}(\F^2)}^2)^2 (\sum_{\tau_N} \|F_{\tau_N}\|_{L^2(\F^2)}^2)
\end{equation}
and then trivially decouple from frequency scale $R_N^{-1/2}$ down to $R^{-1/2}$ (note $R_N^{-1/2} / R^{-1/2} \leq q^{r/2} \lsm (\log R)^{\varepsilon/2}$ which implies $\|F_{\tau_N}\|_{L^{\infty}}^2 \lsm (\log R)^{\varepsilon/2} \sum_{|\tau|=R^{-1/2}} \nms{F_{\tau}}_{L^{\infty}}^{2}$ and $\sum_{\tau_N}\nms{F_{\tau_N}}_{L^{2}}^{2} = \sum_{|\tau|=R^{-1/2}}\nms{F_{\tau}}_{L^{2}}^{2}$ by Plancherel).

\section{Bilinearization}\label{whitney}

The proof of Theorem~\ref{thm:L2Linfty_intro} relies on the following key bilinear estimate:

\begin{prop} \label{prop:bilinear}
Let $F$ be Fourier supported in $\Xi_{1/R}$. For $k = 0, 1, \dots, N-1$, and for intervals $\tau_k \subset \O$ with $|\tau_k| = R_k^{-1/2}$, we have
\begin{equation*}
\int_{\F^2} \max_{\substack{ \tau_{k+1} \ne \tau_{k+1}' \\ \tau_{k+1}, \tau_{k+1}' \subset \tau_k }} |F_{\tau_{k+1}} F_{\tau_{k+1}'}|^3 \lesssim (\log R)^{9 + 6 \varepsilon} (\sum_{\tau_N \subset \tau_k} \|F_{\tau_N}\|_{L^{\infty}(\F^2)}^2)^2 (\sum_{\tau_N \subset \tau_k} \|F_{\tau_N}\|_{L^2(\F^2)}^2).
\end{equation*}
\end{prop}

We also need the following Whitney decomposition for $\O^2$, which expresses $\O^2$ into a disjoint union of squares of different scales:
\begin{equation*}
\O^2 = \mathcal{W}_0 \sqcup \mathcal{W}_1 \sqcup \dots \sqcup \mathcal{W}_{N-1} \sqcup \mathcal{W}^N
\end{equation*}
where
\begin{equation*}
\mathcal{W}_k := \bigsqcup_{\tau_k \subset \O} \bigsqcup_{\substack{\tau_{k+1} \ne \tau_{k+1}' \\ \tau_{k+1}, \tau_{k+1}' \subset \tau_k}} \tau_{k+1} \times \tau_{k+1}' \quad \text{for $k = 0, 1, \dots, N-1$}
\end{equation*}
and
\begin{equation*}
\mathcal{W}^N := \bigsqcup_{\tau_N \subset \O}  \tau_N \times \tau_N.
\end{equation*}

The proof of \eqref{eq:goal0}, and hence Theorem~\ref{thm:L2Linfty_intro} can then be given as follows.
First,
\begin{align*}
\int_{\F^2} |F|^6 = \int_{\F^2} | F^2 |^3
= \int_{\F^2} \Big| \sum_{\tau_N \subset \O} F_{\tau_N}^2 + \sum_{k=0}^{N-1} \sum_{\tau_{k+1} \times \tau_{k+1}' \subset \mathcal{W}_k} F_{\tau_{k+1}} F_{\tau_{k+1}'} \Big|^3
\end{align*}
which by the Minkowski inequality is
\begin{align} \label{eq:5.6}
\leq \left[ \sum_{\tau_N} \left( \int_{\F^2} \Big| F_{\tau_N}^2 \Big|^3 \right)^{1/3} + \sum_{k=0}^{N-1} \sum_{\tau_{k+1} \times \tau_{k+1}' \subset \mathcal{W}_k} \left( \int_{\F^2} \Big| F_{\tau_{k+1}} F_{\tau_{k+1}'} \Big|^3 \right)^{1/3} \right]^3.
\end{align}
H\"{o}lder's inequality gives
\begin{align*}
\sum_{\tau_N} \left( \int_{\F^2} \Big| F_{\tau_N}^2 \Big|^3 \right)^{1/3}
&= \sum_{\tau_N} \| F_{\tau_N} \|_{L^6(\F^2)}^2
\leq \sum_{\tau_N} \| F_{\tau_N} \|_{L^{\infty}(\F^2)}^{2 \cdot \frac{2}{3}} \| F_{\tau_N} \|_{L^2(\F^2)}^{2 \cdot \frac{1}{3}} \\
&\leq ( \sum_{\tau_N} \| F_{\tau_N} \|_{L^{\infty}(\F^2)}^2 )^{\frac{2}{3}} (\sum_{\tau_N} \| F_{\tau_N} \|_{L^2(\F^2)}^2 )^{\frac{1}{3}}.
\end{align*}
In addition, for each fixed $\tau_k$, the number of $(\tau_{k+1},\tau_{k+1}')$ with $\tau_{k+1},\tau_{k+1}' \subset \tau_k$ is $\leq (q^{r/2})^2 \lesssim (\log R)^{\varepsilon}$. Together with Proposition~\ref{prop:bilinear}, this shows that for each $k = 0, 1, \dots, N-1$,
\begin{align*}
&\sum_{\tau_{k+1} \times \tau_{k+1}' \subset \mathcal{W}_k} \left( \int_{\F^2} \Big| F_{\tau_{k+1}} F_{\tau_{k+1}'} \Big|^3 \right)^{1/3}
= \, \sum_{\tau_k} \sum_{ \substack{ \tau_{k+1} \ne \tau_{k+1}' \\ \tau_{k+1}, \tau_{k+1}' \subset \tau_k }} \left( \int_{\F^2} \Big| F_{\tau_{k+1}} F_{\tau_{k+1}'} \Big|^3 \right)^{1/3} \\
\lesssim &\, (\log R)^{3+2\varepsilon} (\log R)^{\varepsilon} \sum_{\tau_k} ( \sum_{\tau_N \subset \tau_k} \| F_{\tau_N} \|_{L^{\infty}(\F^2)}^2 )^{\frac{2}{3}} (\sum_{\tau_N \subset \tau_k} \| F_{\tau_N} \|_{L^2(\F^2)}^2 )^{\frac{1}{3}} \\
\leq &\, (\log R)^{3 + 3\varepsilon}  ( \sum_{\tau_N} \| F_{\tau_N} \|_{L^{\infty}(\F^2)}^2 )^{\frac{2}{3}} (\sum_{\tau_N} \| F_{\tau_N} \|_{L^2(\F^2)}^2 )^{\frac{1}{3}}
\end{align*}
Thus \eqref{eq:5.6} is bounded by
\begin{equation*}
N^3 (\log R)^{9 + 9\varepsilon}  ( \sum_{\tau_N} \| F_{\tau_N} \|_{L^{\infty}(\F^2)}^2 )^2 (\sum_{\tau_N} \| F_{\tau_N} \|_{L^2(\F^2)}^2 )
\end{equation*}
which proves \eqref{eq:goal0} because $N \lesssim \log R$. %

Proposition~\ref{prop:bilinear} can be proved by parabolic rescaling and the proposition below. That is, we use the next proposition with $J = N-k$ and
\begin{equation}\label{rescale1}
f(x) := \chi(-R_k^{1/2} a x_1 +R_k a^2 x_2) F_{\tau_k}(R_k^{1/2} x_1 - 2 a R_k x_2, R_k x_2)
\end{equation}
where $a$ is an arbitrary point in $\tau_k$. Note that
\begin{equation}\label{rescale2}
\widehat{f}(\xi,\eta) =  R_k^{-3/2} \widehat{F}_{\tau_k}(a + R_k^{-1/2} \xi, a^2 + 2 a R_k^{-1/2} \xi + R_k^{-1} \eta)
\end{equation}
is supported on $\Xi_{R_k / R} \subset \Xi_{1/R_{N-k}}$.
\begin{prop} \label{prop:bilinear_rescaled}
Let $J = 1, \dots, N$ and let $f$ be Fourier supported in $\Xi_{1/R_J}$. Then
\begin{equation*}
\int_{\F^2} \max_{\tau_1 \ne \tau_1'} |f_{\tau_1} f_{\tau_1'}|^3 \lesssim (\log R)^{9 + 6 \varepsilon} (\sum_{\tau_J} \|f_{\tau_J}\|_{L^{\infty}(\F^2)}^2)^2 (\sum_{\tau_J} \|f_{\tau_J}\|_{L^2(\F^2)}^2).
\end{equation*}
\end{prop}

It remains to prove Proposition~\ref{prop:bilinear_rescaled}.

\section{Broad/Narrow decomposition: Proof of Proposition~\ref{prop:bilinear_rescaled}}

The proof of Proposition~\ref{prop:bilinear_rescaled} is via a broad/narrow decomposition. Let $J = 1, \dots, N$ and $f$ be Fourier supported in $\Xi_{1/R_J}$. For $k = 0, 1, \dots, J-1$, and for $\tau_k \subset \O$ with $|\tau_k| = R_k^{-1/2}$, define
\begin{align}
\mc{B}_{\tau_k} = & \{x \in \F^2 \colon |f_{\tau_k}(x)| \leq (\log R) q^{r/2} \max_{\substack{\tau_{k+1} \ne \tau_{k+1}' \\ \tau_{k+1}, \tau_{k+1}' \subset \tau_k}} |f_{\tau_{k+1}}(x) f_{\tau_{k+1}'}(x)|^{1/2} \label{eq:alternative1} \\
& \text{and} \quad (\sum_{\tau_{k+1} \subset \tau_k} |f_{\tau_{k+1}}(x)|^6)^{1/6} \leq (\log R) q^{r/2} \max_{\substack{\tau_{k+1} \ne \tau_{k+1}' \\ \tau_{k+1}, \tau_{k+1}' \subset \tau_k}} |f_{\tau_{k+1}}(x) f_{\tau_{k+1}'}(x)|^{1/2} \}. \label{eq:alternative2}
\end{align}
For $x \notin \mc{B}_{\tau_0}$, we have
\begin{equation} \label{eq:initial_narrow}
 \max_{\tau_1 \ne \tau_1'} |f_{\tau_1}(x) f_{\tau_1'}(x)|^3 \leq \frac{q^{-r/2}}{(\log R)^6} \sum_{\tau_1} |f_{\tau_1}(x)|^6.
\end{equation}
This is because if $x \notin \mc{B}_{\tau_0}$, then either \eqref{eq:alternative1} is violated, in which case
\begin{align*}
\max_{\tau_1 \ne \tau_1'} |f_{\tau_1}(x) f_{\tau_1'}(x)|^3
\leq \frac{q^{-3r}}{(\log R)^6} |f(x)|^6
=  \frac{q^{-3r}}{(\log R)^6} |\sum_{\tau_1} f_{\tau_1}(x)|^6
\leq  \frac{q^{-3r}}{(\log R)^6} q^{5r/2} \sum_{\tau_1} |f_{\tau_1}(x)|^6,
\end{align*}
or \eqref{eq:alternative2} is violated, in which case
\begin{equation*}
 \max_{\tau_1 \ne \tau_1'} |f_{\tau_1}(x) f_{\tau_1'}(x)|^3
 \leq  \frac{q^{-3r}}{(\log R)^6} \sum_{\tau_1} |f_{\tau_1}(x)|^6.
\end{equation*}
Either way \eqref{eq:initial_narrow} holds. Upon splitting the integral in Proposition~\ref{prop:bilinear_rescaled} according to whether $x \in \mc{B}_{\tau_0}$ or not, \eqref{eq:initial_narrow} allows us to obtain
\begin{equation} \label{eq:bn_initial}
\int_{\F^2} \max_{\tau_1 \ne \tau_1'} |f_{\tau_1} f_{\tau_1'}|^3
\leq \int_{\mc{B}_{\tau_0}} \max_{\tau_1 \ne \tau_1'} |f_{\tau_1} f_{\tau_1'}|^3 + \frac{q^{-r/2}}{(\log R)^6} \sum_{\tau_1} \int_{\F^2} |f_{\tau_1}|^6.
\end{equation}
Now observe that if $k = 1, \dots, J-1$ and $|\tau_k| = R_k^{-1/2}$, then
\begin{enumerate}[(a)]
\item for $x \in \mc{B}_{\tau_k}$, we have
\begin{equation} \label{eq:broad_pointwise}
|f_{\tau_k}(x)|^6 \leq (\log R)^6 q^{3r} \max_{\substack{\tau_{k+1}, \tau_{k+1}' \subset \tau_k \\ \tau_{k+1} \ne \tau_{k+1}'}} |f_{\tau_{k+1}}(x) f_{\tau_{k+1}'}(x)|^3;
\end{equation}
\item for $x \notin \mc{B}_{\tau_k}$, we have
\begin{equation} \label{eq:narrow_pointwise}
|f_{\tau_k}(x)|^6 \leq (1-(\log R)^{-1})^{-6} \sum_{\tau_{k+1} \subset \tau_k} |f_{\tau_{k+1}}(x)|^6.
\end{equation}
\end{enumerate}
The estimate \eqref{eq:broad_pointwise} holds because of \eqref{eq:alternative1}. The proof of \eqref{eq:narrow_pointwise} proceeds via the Narrow Lemma:

\begin{lem}[Narrow Lemma] \label{lem:narrow}
Fix $\tau_k \subset \O$ with $|\tau_k| = R_{k}^{-1/2}$.
Suppose $x$ satisfies
\begin{equation*}%
|f_{\tau_k}(x)| > (\log R) q^{r/2}  \max_{\substack{\tau_{k + 1}, \tau_{k + 1}' \subset \tau_k\\\tau_{k+1} \ne \tau_{k+1}'}} |f_{\tau_{k+1}}(x) f_{\tau_{k+1}'}(x)|^{1/2}.
\end{equation*}
Then there exists a $\tau_{k+1} \subset \tau_k$ such that
\begin{equation*}%
|f_{\tau_k}(x)| \leq (1 - (\log R)^{-1})^{-1}|f_{\tau_{k+1}}(x)|.
\end{equation*}
\end{lem}

Indeed, for $x \notin \mc{B}_{\tau_k}$, either \eqref{eq:alternative1} fails, in which case the Narrow Lemma applies, or \eqref{eq:alternative1} holds but \eqref{eq:alternative2} fails, in which case
\begin{equation*}
|f_{\tau_k}(x)| \leq (\log R) q^{r/2} \max_{\substack{\tau_{k+1}, \tau_{k+1}' \subset \tau_k \\ \tau_{k+1} \ne \tau_{k+1}'}} |f_{\tau_{k+1}}(x) f_{\tau_{k+1}'}(x)|^{1/2} \leq (\sum_{\tau_{k+1} \subset \tau_k} |f_{\tau_{k+1}}(x)|^6)^{1/6}.
\end{equation*}
Either way \eqref{eq:narrow_pointwise} holds. From \eqref{eq:broad_pointwise} and \eqref{eq:narrow_pointwise}, we see that for $k = 1, \dots, J-1$ and $|\tau_k| = R_k^{-1/2}$,
\begin{align*}
\frac{q^{-r/2}}{(\log R)^6} (1-(\log R)^{-1})^{-6(k-1)} \int_{\F^2} |f_{\tau_k}|^6
\leq \, & q^{5r/2} (1-(\log R)^{-1})^{-6(k-1)} \int_{\mc{B}_{\tau_k}} \max_{\substack{\tau_{k+1} \ne \tau_{k+1}' \\ \tau_{k+1}, \tau_{k+1}' \subset \tau_k}} |f_{\tau_{k+1}} f_{\tau_{k+1}'}|^3 \\
& + \frac{q^{-r/2}}{(\log R)^6} (1-(\log R)^{-1})^{-6k} \sum_{\tau_{k+1} \subset \tau_k} \int_{\F^2}  |f_{\tau_{k+1}}|^6.
\end{align*}
Summing over $\tau_k$, we get
\begin{align*}
&\frac{q^{-r/2}}{(\log R)^6} (1-(\log R)^{-1})^{-6(k-1)} \sum_{\tau_k} \int_{\F^2} |f_{\tau_k}|^6 \\
\leq \, & q^{5r/2} (1-(\log R)^{-1})^{-6(k-1)} \sum_{\tau_k} \int_{\mc{B}_{\tau_k}} \max_{\substack{\tau_{k+1} \ne \tau_{k+1}' \\ \tau_{k+1}, \tau_{k+1}' \subset \tau_k}} |f_{\tau_{k+1}} f_{\tau_{k+1}'}|^3 \\
& + \frac{q^{-r/2}}{(\log R)^6} (1-(\log R)^{-1})^{-6k} \sum_{\tau_{k+1}} \int_{\F^2}  |f_{\tau_{k+1}}|^6
\end{align*}
for $k = 1, \dots, J-1$. We now apply these successively to the right hand side of \eqref{eq:bn_initial}, starting with $k = 1$ and going all the way up to $k = J-1$. Then
\begin{align*}
\int_{\F^2} \max_{\tau_1 \ne \tau_1'} |f_{\tau_1} f_{\tau_1'}|^3 & \leq \int_{\mc{B}_{\tau_0}} \max_{\tau_1 \ne \tau_1'} |f_{\tau_1} f_{\tau_1'}|^3 \\
& + \sum_{k=1}^{J-1} q^{5r/2} (1-(\log R)^{-1})^{-6(k-1)} \sum_{\tau_k} \int_{\mc{B}_{\tau_k}} \max_{\substack{\tau_{k+1} \ne \tau_{k+1}' \\ \tau_{k+1}, \tau_{k+1}' \subset \tau_k}} |f_{\tau_{k+1}} f_{\tau_{k+1}'}|^3  \\
&+ \frac{q^{-r/2}}{(\log R)^6} (1-(\log R)^{-1})^{-6(J-1)} \sum_{\tau_J} \int_{\F^2} |f_{\tau_J}|^6.
\end{align*}
Since $J \leq N \lesssim \log R$, this gives
\begin{align}
\int_{\F^2} \max_{\tau_1 \ne \tau_1'} |f_{\tau_1} f_{\tau_1'}|^3
& \lesssim q^{5r/2} (\log R) \max_{k=0,\dots,J-1} \sum_{\tau_k} \int_{\mc{B}_{\tau_k}} \max_{\substack{\tau_{k+1} \ne \tau_{k+1}' \\ \tau_{k+1}, \tau_{k+1}' \subset \tau_k}} |f_{\tau_{k+1}} f_{\tau_{k+1}'}|^3 \label{eq:broad_final} \\
&+ \frac{q^{-r/2}}{(\log R)^6} \sum_{\tau_J} \int_{\F^2} |f_{\tau_J}|^6. \label{eq:narrow_final}
\end{align}
Observe that
\begin{equation} \label{eq:narrow_last}
\eqref{eq:narrow_final} \lesssim \frac{q^{-r/2}}{(\log R)^6} (\sum_{\tau_J} \|f_{\tau_J}\|_{L^{\infty}(\F^2)}^2)^2 (\sum_{\tau_J} \|f_{\tau_J}\|_{L^2(\F^2)}^2)
\end{equation}
which is much better than what we needed in the conclusion of Proposition~\ref{prop:bilinear_rescaled}. Equation \eqref{eq:broad_final} is controlled by the following proposition:

\begin{prop} \label{prop:bilinear_broad}
Let $J = 1, \dots, N$ and let $f$ be Fourier supported in $\Xi_{1/R_J}$. Let $k = 0, 1, \dots, J-1$ and $\tau_k \subset \O$ with $|\tau_k| = R_k^{-1/2}$. Then
\begin{equation} \label{eq:bilinear_broad}
\int_{\mc{B}_{\tau_k}} \max_{\substack{\tau_{k+1} \ne \tau_{k+1}' \\ \tau_{k+1}, \tau_{k+1}' \subset \tau_k}} |f_{\tau_{k+1}} f_{\tau_{k+1}'}|^3
\lesssim (\log R)^{8 + \frac{7 \varepsilon}{2}} (\sum_{\tau_J \subset \tau_k} \|f_{\tau_J}\|_{L^{\infty}(\F^2)}^2)^2 (\sum_{\tau_J \subset \tau_k} \|f_{\tau_J}\|_{L^2(\F^2)}^2).
\end{equation}
\end{prop}

Assuming this for the moment, we see that \eqref{eq:broad_final} is bounded by
\begin{align*}
\eqref{eq:broad_final} & \lesssim q^{5r/2} (\log R)^{1+8+\frac{7 \varepsilon}{2}} \max_{k=0,\dots,J-1} \sum_{\tau_k} (\sum_{\tau_J \subset \tau_k} \|f_{\tau_J}\|_{L^{\infty}(\F^2)}^2)^2 (\sum_{\tau_J \subset \tau_k} \|f_{\tau_J}\|_{L^2(\F^2)}^2) \\
& \lesssim  (\log R)^{9+ 6 \varepsilon} (\sum_{\tau_J} \|f_{\tau_J}\|_{L^{\infty}(\F^2)}^2)^2 (\sum_{\tau_J} \|f_{\tau_J}\|_{L^2(\F^2)}^2).
\end{align*}
(Recall $q^{5r/2} \leq (\log R)^{5 \varepsilon/2}$.) Together with \eqref{eq:narrow_last} we finish the proof of Proposition~\ref{prop:bilinear_rescaled}.
It remains to prove Lemma~\ref{lem:narrow} and Proposition~\ref{prop:bilinear_broad}.

\begin{proof}[Proof of Lemma~\ref{lem:narrow}]
Let $\tau_{k + 1}^{\ast}$ be the $\tau_{k + 1} \subset \tau_{k}$ such that
\begin{equation*}
\max_{\tau_{k + 1} \subset \tau_{k}}|f_{\tau_{k + 1}}(x)| = |f_{\tau_{k + 1}^{\ast}}(x)|.
\end{equation*}
For $\tau_{k + 1} \subset \tau_k$ such that $\tau_{k + 1} \neq \tau_{k + 1}^{\ast}$, note that
\begin{equation*}
|f_{\tau_{k + 1}}(x)| \leq |f_{\tau_{k + 1}}(x)f_{\tau_{k + 1}^{\ast}}(x)|^{1/2} < (\log R)^{-1}q^{-r/2}|f_{\tau_{k}}(x)|.
\end{equation*}
Therefore
\begin{align*}
|f_{\tau_{k + 1}^{\ast}}(x)| &= |f_{\tau_k}(x) - \sum_{\tau_{k + 1} \neq \tau_{k + 1}^{\ast}} f_{\tau_{k + 1}}(x)|\\
&\geq (1 - \#\{\tau_{k + 1}: \tau_{k + 1} \subset \tau_k, \tau_{k + 1}\neq \tau_{k + 1}^{\ast}\}(\log R)^{-1}q^{-r/2})|f_{\tau_k}(x)|\\
&\geq (1 - (\log R)^{-1})|f_{\tau_k}(x)|.
\end{align*}
\end{proof}

To prove Proposition~\ref{prop:bilinear_broad}, we need the following level set estimate.

\begin{prop}\label{prop:newprop34}
Let $J = 1, \dots, N$ and let $f$ be with Fourier support in $\Xi_{1/R_J}$. For $\alpha > 0$, let
\begin{align*}%
U_{\alpha}(f) := \{x \in \F^{2}: \max_{\tau_1 \neq \tau_{1}'}|f_{\tau_1}(x)f_{\tau_{1}'}(x)|^{1/2} \sim \alpha \text{ and } (\sum_{\tau_1}|f_{\tau_1}(x)|^{6})^{1/6} \lsm (\log R)q^{r/2}\alpha\}.
\end{align*}
Then
\begin{align*}
\alpha^{6}|U_{\alpha}(f)| \lsm (\log R)^{7 + \frac{7 \vep}{2}}(\sum_{\tau_{J}}\nms{f_{\tau_{J}}}_{L^{\infty}(\F^{2})}^{2})^{2}
(\sum_{\tau_{J}} \nms{f_{\tau_J}}_{L^{2}(\F^{2})}^{2})
\end{align*}
where the implied constant is independent of $f$ and $\alpha$.
\end{prop}

\begin{proof}[Proof of Proposition~\ref{prop:bilinear_broad}]
By the same rescaling as in \eqref{rescale1}-\eqref{rescale2}, it suffices to prove \eqref{eq:bilinear_broad} for $k = 0$.
For a given $J_0 = 1, 2, \ldots, N$ and $k_0 = 1, 2, \ldots, J_0 - 1$,
the case of $(k, J) = (k_0, J_0)$ in \eqref{eq:bilinear_broad} follows from the case
$(k, J) = (0, J_0 - k_0)$. Note also that in this rescaling, it is important
that in the definition of $\mc{B}_{\tau_k}$ we have the condition $x \in \F^2$ in \eqref{eq:alternative1} rather than a smaller spatial region.

Now to prove \eqref{eq:bilinear_broad} for $k = 0$, for each square $Q_{R_J^{1/2}} \subset \F^2$ of side length $R_J^{1/2}$, we estimate
\begin{equation}\label{eq:br0dom1}
\int_{\mc{B} \cap Q_{R_J^{1/2}}}\max_{\tau_1 \neq \tau_{1}'}|f_{\tau_1}(x)f_{\tau_{1}'}(x)|^{3}
\end{equation}
where we write $\mc{B} := \mc{B}_{\tau_0}$ for brevity.
Let
\[
\mc{B}_{\text{small}}(Q_{R_J^{1/2}}) := \{x \in \mc{B} \cap Q_{R_J^{1/2}} : \max_{\tau_1 \neq \tau_{1}'}|f_{\tau_1}(x)f_{\tau_{1}'}(x)|^{1/2} \leq R^{-1/2}\max_{\tau_{J}}\nms{f_{\tau_{J}}}_{L^{\infty}(Q_{R_J^{1/2}})}\}
\]
and partition $(\mc{B} \cap Q_{R_J^{1/2}}) \bs \mc{B}_{\text{small}}(Q_{R_J^{1/2}})$ into $O(\log R)$ sets where
\begin{equation*}
\max_{\tau_1 \neq \tau_{1}'}|f_{\tau_1}(x)f_{\tau_{1}'}(x)|^{1/2}\sim \alpha
\quad \text{and} \quad
R^{-1/2}\max_{\tau_J} \nms{f_{\tau_{J}}}_{L^{\infty}(Q_{R_J^{1/2}})} \leq \alpha \leq R\max_{\tau_J} \nms{f_{\tau_{J}}}_{L^{\infty}(Q_{R_J^{1/2}})}.
\end{equation*}
By pigeonholing, there exists an $\alpha_{\ast}$ such that
\begin{align} \label{3.6}
 \eqref{eq:br0dom1} \lsm R_J R^{-3} \max_{\tau_{J}}\nms{f_{\tau_{J}}}_{L^{\infty}(Q_{R_J^{1/2}})}^{6} + (\log R)\alpha_{\ast}^{6}|Q_{R_{J}^{1/2}} \cap U_{\alpha_{\ast}}(f)|.
\end{align}
But by the uncertainty principle (see discussion after Lemma~\ref{lem:uncertainty}), $|f_{\tau_J}|$ is constant on $Q_{R_J^{1/2}}$, so
\begin{equation*}
\|f_{\tau_J}\|_{L^{\infty}(Q_{R_J^{1/2}})}^2 = R_J^{-1} \|f_{\tau_J}\|_{L^2(Q_{R_J^{1/2}})}^2 \leq \|f_{\tau_J}\|_{L^2(Q_{R_J^{1/2}})}^2.
\end{equation*}
Thus
\begin{equation*}
\max_{\tau_J} \|f_{\tau_J}\|_{L^{\infty}(Q_{R_J^{1/2}})}^6 \leq \max_{\tau_J} \|f_{\tau_J}\|_{L^{\infty}(\F^2)}^4 \|f_{\tau_J}\|_{L^2(Q_{R_J^{1/2}})}^2 \leq ( \sum_{\tau_J} \|f_{\tau_J}\|_{L^{\infty}(\F^2)}^2 )^2 \sum_{\tau_J} \|f_{\tau_J}\|_{L^2(Q_{R_J^{1/2}})}^2.
\end{equation*}
Plugging this back into \eqref{3.6}, and summing over $Q_{R_J^{1/2}}$, we obtain
\begin{align*}
\int_{\mc{B}}\max_{\tau_1 \neq \tau_{1}'}|f_{\tau_1}(x)f_{\tau_{1}'}(x)|^{3}
&\lesssim R_J R^{-3} ( \sum_{\tau_J} \|f_{\tau_J}\|_{L^{\infty}(\F^2)}^2 )^2 \sum_{\tau_J} \|f_{\tau_J}\|_{L^2(\F^2)}^2 + (\log R) \alpha_*^6 |U_{\alpha_*}(f)| \\
&\lesssim (\log R)^{8+ \frac{7 \varepsilon}{2}} ( \sum_{\tau_J} \|f_{\tau_J}\|_{L^{\infty}(\F^2)}^2 )^2 \sum_{\tau_J} \|f_{\tau_J}\|_{L^2(\F^2)}^2
\end{align*}
where the last inequality is a consequence of Proposition~\ref{prop:newprop34}. This finishes our proof.
\end{proof}

The rest of the argument goes into proving Proposition~\ref{prop:newprop34}.

\section{High/Low decomposition: Proof of Proposition~\ref{prop:newprop34}}\label{highlow}

\subsection{Square functions and pruning of wave packets}

Fix $J = 1, \dots, N$ and fix $f$ with Fourier support in $\Xi_{1/R_J}$.
For $x \in \F^2$ and $\ld$ to be chosen later (see \eqref{eq:lambda}), define
\begin{align*}
g_{J}(x) &:= \sum_{\tau_{J}}|f_{\tau_{J}}(x)|^{2} = \sum_{\tau_{J}}\sum_{T_{J} \in \T(\tau_{J})}|(1_{T_{J}}f_{\tau_{J}})(x)|^{2}\\
f_{J}(x) &:= \sum_{\tau_{J}}\sum_{\st{T_{J} \in \T(\tau_J)\\\|1_{T_{J}}f_{\tau_{J}}\|_{L^{\infty}(\F^{2})} \leq \ld}}(1_{T_{J}}f_{\tau_{J}})(x)
\end{align*}
and for $k = J - 1, J - 2, \ldots, 1$, define
\begin{align*}
g_{k}(x)&:= \sum_{\tau_k}|(f_{k + 1, \tau_k})(x)|^{2} = \sum_{\tau_k}\sum_{T_k \in \T(\tau_k)}|(1_{T_k}f_{k + 1, \tau_k})(x)|^{2}\\
f_{k}(x) &:= \sum_{\tau_k}\sum_{\st{T_k \in \T(\tau_k)\\\|1_{T_k}f_{k + 1, \tau_k}\|_{L^{\infty}(\F^{2})} \leq \ld}}(1_{T_k}f_{k + 1, \tau_k})(x).
\end{align*}
Note that the Fourier support of $g_k$ is contained in a $R_{k}^{-1/2}$ square centered at the origin
and hence $g_k$ is constant on squares of side length $R_{k}^{1/2}$.
Additionally by definition of the $f_k$,
\begin{align}\label{eq:prunedmon}
|f_{k, \tau_k}| \leq |f_{k + 1, \tau_k}|
\end{align}
and so
\begin{align*}
\int_{\F^{2}}|f_{k}|^{2} = \sum_{\tau_k}\int_{\F^{2}}|f_{k, \tau_k}|^{2} \leq \sum_{\tau_k}\int_{\F^{2}}|f_{k + 1, \tau_k}|^{2} = \int_{\F^{2}}|f_{k + 1}|^{2},
\end{align*}
where in the last step we applied $L^2$ orthogonality. Therefore
\begin{align}\label{eq:increase}
\int_{\F^{2}}|f_1|^{2} \leq \int_{\F^{2}}|f_2|^{2} \leq \cdots \leq \int_{\F^{2}}|f_{J}|^{2} \leq \int_{\F^{2}}|f|^{2}.
\end{align}
This matches the intuition that when passing from $f_{J}$ to $f_1$ we are throwing away wave packets and therefore at least at the $L^2$ level, we have a monotonicity relation as above.

\subsection{High and low lemmas}

For $k = 1, \dots, J-1$, define
\[
g_k^l = g_k \ast R_{k + 1}^{-1}1_{B(0, R_{k + 1}^{1/2})} \quad \text{and} \quad g_k^h = g_k - g_k^l.
\]
Note that $g_k$ (and $g_{k}^h$) is Fourier supported on the union of  $\{|\xi| \leq R_k^{-1/2}, |\eta - 2\alpha \xi| \leq R_k^{-1}\}$ where $\{\alpha\}$ is a collection
of points chosen from $\{\tau_k\}$, with one $\alpha$ for each $\tau_k$.
Additionally, observe that since
\begin{align}\label{qrk1ft}
R_{k + 1}^{-1}\widehat{1}_{B(0, R_{k + 1}^{1/2})} = 1_{B(0, R_{k + 1}^{-1/2})}
\end{align}
we have
$\wh{g_{k}^{l}} = \wh{g_k}1_{B(0, R_{k + 1}^{-1/2})}$ and so $g_{k}^{l}$ is just the restriction
of $g_k$ to frequencies less than $R_{k + 1}^{-1/2}$.
By definition of $g_k$ and $g_{k}^{l}$, both are nonnegative functions.

\begin{lem}[Low Lemma]\label{201218lem4_4}
For $k = 1, \dots, J-1$, we have $g_k^l \leq g_{k+1}$.
\end{lem}

\begin{proof}[Proof of Lemma \ref{201218lem4_4}]
We have
\begin{align}\label{lowlemeq1}
g_{k}^{l} = g_{k} \ast R_{k + 1}^{-1}1_{B(0, R_{k + 1}^{1/2})} = \sum_{\tau_k}\sum_{\tau_{k + 1}, \tau_{k + 1}' \subset \tau_k}(f_{k + 1, \tau_{k + 1}}\ov{f_{k + 1, \tau_{k + 1}'}})\ast R_{k + 1}^{-1}1_{B(0, R_{k + 1}^{1/2})}.
\end{align}
Taking a Fourier transform we see that
\begin{align*}
(f_{k + 1, \tau_{k + 1}}\ov{f_{k + 1, \tau_{k + 1}'}})\ast R_{k + 1}^{-1}1_{B(0, R_{k + 1}^{1/2})} &= \begin{cases}
|f_{k + 1, \tau_{k + 1}}|^{2} \ast R_{k + 1}^{-1}1_{B(0, R_{k + 1}^{1/2})} & \text{ if } \tau_{k + 1} = \tau_{k + 1}'\\
0 & \text{ otherwise}
\end{cases}\\
&=
\begin{cases}
|f_{k + 1, \tau_{k + 1}}|^{2}  & \text{ if } \tau_{k + 1} = \tau_{k + 1}'\\
0 & \text{ otherwise}
\end{cases}
\end{align*}
where the last equality is because of \eqref{qrk1ft} and that $|f_{k + 1, \tau_{k + 1}}|^{2}$
is Fourier supported in $B(0, R_{k + 1}^{-1/2})$.
Thus \eqref{lowlemeq1} is equal to
\begin{align*}
\sum_{\tau_{k + 1}}|f_{k + 1, \tau_{k + 1}}|^{2} \leq \sum_{\tau_{k + 1}}|f_{k + 2, \tau_{k + 1}}|^{2} = g_{k + 1}
\end{align*}
by \eqref{eq:prunedmon}. Here if $k = J - 1$, we interpret $f_{k + 2}$ to mean $f$.
\end{proof}

\begin{lem}[High Lemma]\label{201218lem4_5} For $k = 1, \dots, J-1$,
$$\int_{\F^2} |g_k^h|^2 \leq q^{r/2} \sum_{\tau_k}\int_{\F^2}|f_{k + 1, \tau_k}|^{4}. $$
\end{lem}

\begin{proof}[Proof of Lemma \ref{201218lem4_5}]
It suffices to partition $\F^2$ into squares with side length $R_{k + 1}$ and prove the estimate on each such square.
Fix an arbitrary square $B \subset \F^2$ of side length $R_{k + 1}$.
We have by Plancherel,
\begin{align*}
\int_{B}|g_{k}^{h}|^{2} = \int \ov{\wh{g_{k}^h}}(\wh{g_{k}^h} \ast \wh{1_B}).
\end{align*}
Since $g_{k}^{h}$ is Fourier supported outside $B(0, R_{k + 1}^{-1/2})$
and $1_{B}$ is Fourier supported in $B(0, R_{k + 1}^{-1})$, $\wh{g_{k}^h} \ast \wh{1_B}$
is supported in $B(0, R_{k}^{-1/2}) \setminus B(0, R_{k + 1}^{-1/2})$ by the ultrametric inequality.
Therefore the above is equal to
\begin{align}\label{higheq1}
\sum_{\tau_k}\int_{B(0, R_{k}^{-1/2}) \setminus B(0, R_{k + 1}^{-1/2})}\ov{(|f_{k + 1, \tau_k}|^{2})^{\wedge}}\sum_{\tau_{k}'}((|f_{k + 1, \tau_{k}'}|^{2})^{\wedge} \ast \wh{1_B}).
\end{align}
We claim that for each $\tau_k$, the Fourier support of $|f_{k + 1, \tau_k}|^{2}$ outside $B(0, R_{k + 1}^{-1/2})$ only intersects $q^{r/2}$ many
Fourier supports of the $|f_{k + 1, \tau_{k}'}|^{2}$ outside $B(0, R_{k + 1}^{-1/2})$.

Indeed, suppose there exists $(\xi,\eta)$ such that $\max\{|\xi|,|\eta|\} > R_{k+1}^{-1/2}$
and
\[
|\xi| \leq R_k^{-1/2}, \qquad |\eta-2\alpha\xi|, |\eta-2\alpha'\xi| \leq R_k^{-1}
\]
for some $\alpha \in \tau_k$ and $\alpha' \in \tau_{k}'$.
Then
\[
|2(\alpha-\alpha')\xi| \leq R_k^{-1},
\]
and so if $|\xi| > R_{k+1}^{-1/2}$, then
\[
|\alpha-\alpha'| \leq R_k^{-1} / R_{k+1}^{-1/2} = R_k^{-1/2} q^{r/2}.
\]
Else $|\xi| < R_{k+1}^{-1/2}$ and $|\eta| > R_{k+1}^{-1/2}$, which implies $|\eta - 2\alpha \xi| = \max\{|\eta|, |2\alpha \xi|\} > R_{k+1}^{-1/2}$,
contradicting $|\eta - 2\alpha \xi| \leq R_k^{-1}$ if $k \geq 1$. So $|\alpha-\alpha'| \leq R_k^{-1/2} q^{r/2}$, the number of overlaps is just $q^{r/2}$ times.

Thus we have
\begin{align*}
\sum_{\tau_k}&\int_{B(0, R_{k}^{-1/2})\setminus B(0, R_{k + 1}^{-1/2})}\ov{(|f_{k + 1, \tau_k}|^{2})^{\wedge}}\sum_{\tau_{k}': d(\tau_k, \tau_{k}') \leq R_{k}^{-1/2}q^{r/2}}(|f_{k + 1, \tau_{k}'}|^{2})^{\wedge} \ast \wh{1_B}\\
&=\sum_{\tau_k}\int_{B}|f_{k + 1, \tau_k}|^{2}\ast (\widecheck{1}_{B(0, R_{k}^{-1/2})} - \widecheck{1}_{B(0, R_{k}^{-1/2})})\sum_{\tau_{k}': d(\tau_k, \tau_{k}') \leq R_{k}^{-1/2}q^{r/2}}|f_{k + 1, \tau_{k}'}|^{2}\\
&\leq \sum_{\tau_k}\int_{B}|f_{k + 1, \tau_k}|^{2}\sum_{\tau_{k}': d(\tau_k, \tau_{k}') \leq R_{k}^{-1/2}q^{r/2}}|f_{k + 1, \tau_{k}'}|^{2}
\end{align*}
where in the last inequality we have used that $|f_{k + 1, \tau_k}|^{2} \ast \widecheck{1}_{B(0, R_{k}^{-1/2})} = |f_{k + 1, \tau_k}|^{2}$,
$\widecheck{1}_{B(0, R_{k + 1}^{-1/2})}$ is nonnegative, and that the convolution of two nonnegative functions is also nonnegative.
Applying Cauchy-Schwarz then gives that \eqref{higheq1} is
\begin{align*}
\leq q^{r/2}\sum_{\tau_k}\int_{B}|f_{k + 1, \tau_k}|^{4}
\end{align*}
and summing over all $B \subset \F^2$ of side length $R_{k + 1}$ then completes the proof.
\end{proof}

\subsection{Decomposition into high and low sets}\label{7.3decomp}

Let
\begin{equation*}
\Omega_{J-1} = \{x \in \F^2 \colon g_{J-1}(x) \leq (\log R) g_{J-1}^h(x)\}
\end{equation*}
For $k = J-2, J-3, \dots, 1$, define
\begin{align*} %
\Omega_k = \{x \in \F^2 \setminus (\Omega_{k+1} \cup \dots \cup \Omega_{J-1}) \colon g_k(x) \leq (\log R) g_k^h(x)\}
\end{align*}
Finally,
\begin{equation*}
L = \F^2 \setminus (\Omega_1 \cup \dots \cup \Omega_{J-1}).
\end{equation*}

Note that $g_k$ is constant on squares of size $R_{k}^{1/2}$. By definition,
$g_{k}^{l}$ is constant on squares of size $R_{k + 1}^{1/2} > R_{k}^{1/2}$.
Therefore $g_{k}^{h}$ is also constant on squares of size $R_{k}^{1/2}$.

One can view the construction of the $\Om_k$ as follows.
Partition $\F^2$ first into squares of size $R_{J  - 1}^{1/2}$.
Then $\Om_{J - 1}$ is a union of those squares on which $g_{J - 1}(x) \leq (\log R) g_{J - 1}^{h}(x)$
where here we have used that both $g_{J - 1}$ and $g_{J - 1}^{h}$ are constant on each such square of size $R_{J - 1}^{1/2}$.

Next, partition each of the remaining squares not chosen to be part of $\Om_{J - 1}$
into squares of size $R_{J - 2}^{1/2}$.
From these squares of size $R_{J - 2}^{1/2}$, $\Om_{J - 2}$ is the union of those
squares on which $g_{J - 2}(x) \leq (\log R)g_{J - 2}^{h}(x)$.
Repeat this until we have defined $\Om_{1}$ after which we call
the remaining set $L$ (which can be written as the union of squares of size $R_{1}^{1/2}$).

To prove Proposition~\ref{prop:newprop34}, note that
\begin{align} \label{eq:goal2}
\alpha^{6}|U_{\alpha}(f)| \leq \alpha^{6}|U_{\alpha}(f) \cap L| + \sum_{k = 1}^{J - 1}\alpha^{6}|U_{\alpha}(f) \cap \Omega_{k}|.
\end{align}
In view of the definition of the set $U_{\alpha}(f)$, to control the right hand side, we need to understand the size of $\max_{\tau_1 \ne \tau_1'} |f_{\tau_1}(x) f_{\tau_1'}(x)|$ on $\Omega_k$ (for $k = 1, \dots, J-1$) and on $L$. We do so in the next section, and then use it to bound the right hand side of \eqref{eq:goal2}.

\subsection{Approximation by pruned wave packets}
\begin{lem}\label{lem:3.22}
Let $k = 1, 2, \ldots, J - 1$ and $|\tau| \geq R_{k}^{-1/2}$. Then for $x \in \F^2$,
\begin{align*}%
|\sum_{\tau_k \subset \tau}f_{k + 1, \tau_k}(x) - \sum_{\tau_k \subset \tau}f_{k, \tau_k}(x)| \leq \ld^{-1}g_{k}(x).
\end{align*}
\end{lem}
\begin{proof}[Proof of Lemma \ref{lem:3.22}]
Fix $x \in \F^2$. We have
\begin{align}
|\sum_{\tau_k \subset \tau}f_{k + 1, \tau_{k}}(x) - \sum_{\tau_{k} \subset \tau}f_{k, \tau_k}(x)| &= |\sum_{\tau_k \subset \tau}\sum_{\substack{T_k \in \T(\tau_k)\\\|1_{T_k}f_{k + 1, \tau_k}\|_{L^{\infty}(\F^2)} > \ld}}(1_{T_k}f_{k + 1, \tau_k})(x)|\nonumber\\
&\leq \sum_{\tau_k \subset \tau}\sum_{\substack{T_k \in \T(\tau_k)\\\|1_{T_k}f_{k + 1, \tau_k}\|_{L^{\infty}(\F^2)} > \ld}}|(1_{T_k}f_{k + 1, \tau_k})(x)|\label{lem73eq1}.
\end{align}
For each $\tau_k$, there exists exactly a parallelogram $\mc{T}_{k}(x)$ depending on $x$ in $\T(\tau_k)$ such that
$x \in \mc{T}_{k}(x)$. If for this parallelogram, $\nms{1_{\mc{T}_{k}(x)}f_{k + 1, \tau_k}}_{L^{\infty}(\F^2)} \leq \ld$, then
the inner sum for this particular $\tau_k$ in \eqref{lem73eq1} is equal to 0.
Otherwise,
\begin{align*}
|(1_{T_k}f_{k + 1, \tau_k})(x)| \leq \frac{\nms{1_{\mc{T}_{k}(x)}f_{k + 1, \tau_k}}_{L^{\infty}(\F^2)}^{2}}{\ld}
\end{align*}
and hence
\begin{align*}
\sum_{\substack{T_k \in \T(\tau_k)\\\|1_{T_k}f_{k + 1, \tau_k}\|_{L^{\infty}(\F^2)} > \ld}}|(1_{T_k}f_{k + 1, \tau_k})(x)| \leq \ld^{-1}\nms{1_{\mc{T}_{k}(x)}f_{k + 1, \tau_k}}_{L^{\infty}(\F^2)}^{2}.
\end{align*}
Since $|f_{k + 1, \tau_k}|$ is constant on $\mc{T}_{k}(x)$,
$\|1_{\mc{T}_{k}(x)}f_{k + 1, \tau_k}\|_{L^{\infty}(\F^2)}^{2} = |(1_{\mc{T}_{k}(x)}f_{k + 1, \tau_k})(x)|^{2}$
and so \eqref{lem73eq1} is
\begin{align*}
\leq \ld^{-1} \sum_{\tau_k \subset \tau}\sum_{\substack{T_k \in \T(\tau_k)\\\|1_{T_k}f_{k + 1, \tau_k}\|_{L^{\infty}(\F^2)} > \ld}}|(1_{T_k}f_{k + 1, \tau_k})(x)|^{2}\leq \ld^{-1}g_{k}(x)
\end{align*}
which completes the proof of the lemma.
\end{proof}

\begin{lem}\label{lem:3.27}
Let $k = 1, 2, \ldots, J - 1$ and $|\tau| \geq R_{k}^{-1/2}$. Then for $x \in \Om_k$,
\begin{align*}%
|f_{\tau}(x) - \sum_{\tau_k \subset \tau}f_{k + 1, \tau_k}(x)| \lsm \ld^{-1}\frac{\log R}{\log\log R}\nms{g_{J}}_{L^{\infty}(\F^2)}.
\end{align*}
\end{lem}
\begin{proof}[Proof of Lemma \ref{lem:3.27}]
Fix $x \in \Om_k$.
Since $\sum_{\tau_k \subset \tau}f_{\tau_k} = f_{\tau} = \sum_{\tau_{k- 1} \subset \tau}f_{\tau_{k - 1}}$, we have
\begin{align*}
|f_{\tau}(x) - \sum_{\tau_k \subset \tau}f_{k + 1, \tau_k}(x)| & \leq |f_{\tau}(x) - \sum_{\tau_{J} \subset \tau}f_{J, \tau_{J}}(x)| + \sum_{j = k + 1}^{J - 1}|\sum_{\tau_{j} \subset \tau}f_{j + 1, \tau_j}(x) - \sum_{\tau_j \subset \tau}f_{j, \tau_j}(x)|\\
 &\leq \ld^{-1}\sum_{j = k + 1}^{J}g_{j}(x)
\end{align*}
by Lemma \ref{lem:3.22} (by how $f_{J}$ is defined, the $f_{\tau} - \sum_{\tau_{J} \subset \tau}f_{J, \tau_{J}}$ term is controlled by the same proof as in Lemma \ref{lem:3.22}).

To control this sum, we now use the definition of $\Om_k$.
The low lemma gives
$$g_{j}(x) = g_{j}^{l}(x) + g_{j}^{h}(x) \leq g_{j + 1}(x) + g_{j}^{h}(x).$$
Since $x \in \Om_k$, for $j = k + 1, \ldots, J - 1$, this is then $\leq g_{j + 1}(x) + (\log R)^{-1} g_{j}(x)$
and hence
\begin{align*}
g_{j}(x) \leq  (1-(\log R)^{-1})^{-1} g_{j + 1}(x).
\end{align*}
Therefore for $j = k + 1, \ldots, J - 1$,
\begin{align*}
g_{j}(x) \leq (1-(\log R)^{-1})^{-(J - j)} \|g_{J}\|_{L^{\infty}(\F^2)}.
\end{align*}
Thus
\begin{align*}
\ld^{-1}\sum_{j = k + 1}^{J} g_{j}(x) &\leq \ld^{-1}\|g_{J}\|_{L^{\infty}(\F^2)}\sum_{j = k + 1}^{J} (1-(\log R)^{-1})^{-(J - j)}\\
& \lsm \ld^{-1}\frac{\log R}{\log \log R}\|g_{J}\|_{L^{\infty}(\F^2)}
\end{align*}
which completes the proof of Lemma \ref{lem:3.27}.
\end{proof}
Note that the above proof also works for $x \in L$ and we obtain the same conclusion.

Now choose
\begin{align}\label{eq:lambda}
\ld := (\log R)^2 q^{r/2} \frac{\|g_{J}\|_{L^{\infty}(\F^2)}}{\alpha}.
\end{align}
We can write the conclusion of Lemma \ref{lem:3.27} as for $x \in \Om_k$ and $|\tau| \geq R_{k}^{-1/2}$, we have
\begin{align*}
f_{\tau}(x) = f_{k + 1, \tau}(x) + O((\log R)^{-1} q^{-r/2} (\log \log R)^{-1}\alpha)
\end{align*}
and so for $x \in \Om_k$ and $\tau_1, \tau_{1}'$ disjoint intervals of length $R_{1}^{-1/2}$,
\begin{align*}
|f_{\tau_1}(x)f_{\tau_{1}'}(x)| = &|f_{k + 1, \tau_{1}}(x)f_{k + 1, \tau_{1}'}(x)|\\
& + O\bigg(\frac{\alpha}{(\log R) q^{r/2} \log\log R}(|f_{\tau_{1}}(x)| + |f_{\tau_{1}'}(x)|) + \frac{\alpha^2}{(\log R)^{2} q^r (\log\log R)^2}\bigg).
\end{align*}
Since $x \in U_{\alpha}(f)$, we control the $|f_{\tau_1}(x)|$ and $|f_{\tau_{1}'}(x)|$ by the $l^6$ sum over all such $\tau_1$ caps and thus by $(\log R)q^{r/2}\alpha$.
This gives that for $x \in U_{\alpha}(f) \cap \Om_k$,
\begin{align*}
|f_{\tau_1}(x)f_{\tau_{1}'}(x)| = |f_{k + 1, \tau_1}(x)f_{k + 1, \tau_{1}'}(x)| + O(\frac{\alpha^2}{\log\log R}).
\end{align*}
This implies for $x \in U_{\alpha}(f) \cap \Om_k$ and $R$ sufficiently large,
\begin{align*}
\max_{\tau_1 \neq \tau_{1}'}|f_{\tau_1}(x)f_{\tau_{1}'}(x)|^{2} \lsm \max_{\tau_{1} \neq \tau_{1}'}|f_{k + 1, \tau_1}(x)f_{k + 1, \tau_{1}'}(x)|^{2}
\end{align*}
which gives
\begin{align}\label{eq:pruned}
\alpha^{4}|U_{\alpha}(f) \cap \Om_k| \lsm \nms{\max_{\tau_{1} \neq \tau_{1}'}|f_{k + 1, \tau_1}(x)f_{k + 1, \tau_{1}'}(x)|^{1/2}}_{L^{4}(U_{\alpha}(f) \cap \Om_k)}^{4}.
\end{align}

Similarly, Lemma \ref{lem:3.22} with $k = 1$ implies $|f_{2, \tau_1}(x) - f_{1, \tau_{1}}(x)| \leq \ld^{-1}g_{1}(x)$
and the beginning of the proof of Lemma \ref{lem:3.27} implies
$|f_{\tau_1}(x) - f_{2, \tau_1}(x)| \leq \ld^{-1}\sum_{j = 2}^{J} g_{j}(x)$.
Following the proof of Lemma \ref{lem:3.27} and the choice of $\ld$ in \eqref{eq:lambda}
shows that for $x \in L$,
\begin{align*}
f_{\tau_1}(x) = f_{1, \tau_{1}}(x) + O((\log R)^{-1} q^{-r/2} (\log\log R)^{-1}\alpha)
\end{align*}
from which following the same reasoning as in the $\Om_k$ case, we obtain that
\begin{align}\label{eq:prunedlow}
\alpha^{6}|U_{\alpha}(f) \cap L| \lsm \nms{\max_{\tau_{1} \neq \tau_{1}'}|f_{1, \tau_{1}}(x)f_{1, \tau_{1}'}(x)|^{1/2}}_{L^{6}(U_{\alpha}(f) \cap L)}^{6}.
\end{align}
In light of \eqref{eq:goal2}, it remains to estimate the right hand sides of \eqref{eq:pruned} and \eqref{eq:prunedlow}.

\subsection{Estimating $\alpha^6 |U_{\alpha}(f) \cap \Omega_k|$ for $k = 1, \dots, J-1$}

We first recall the following bilinear restriction theorem whose proof we defer to the end of this section.

\begin{lem}[Bilinear restriction] \label{lem_bilinear}
Suppose $\delta \in q^{-2\N}$, and for $i = 1,2$, $f_i$ is a function on $\F^2$ whose Fourier support is contained in
$ %
\{(\xi,\eta) \colon \xi \in I_i, |\eta - \xi^2| \leq \delta\},
$ %
where $I_1$, $I_2$ are intervals in $\O$ (not necessarily of the same length) separated by a distance~$\kappa$. Assume
\begin{equation} \label{assump}
\kappa \geq \delta^{1/2}.
\end{equation}
Then
\begin{equation} \label{bilrest}
\int_{\F^2} |f_1 f_2|^2 \leq \frac{\delta^2}{\kappa} \int_{\F^2} |f_1|^2 \int_{\F^2} |f_2|^2.
\end{equation}
\end{lem}

Fix $k = 1, 2, \dots, J - 1$ below. Then \eqref{eq:pruned} is  bounded by
\begin{equation} \label{eq:pruned2}
\sum_{\tau_1 \ne \tau_1'} \int_{\Omega_k} |f_{k+1,\tau_1} f_{k+1,\tau_1'}|^2.
\end{equation}
Since $g_k$ and $g_k^h$ are constant on squares of side length $R_{k}^{1/2}$, we may partition $\Om_k$ into squares $Q$ of side length $R_{k}^{1/2}$, and integrate on each such $Q$ before we sum over $Q$. If $k \ge 2$, then the Fourier supports of $f_{k+1,\tau_1} 1_Q$ and $f_{k+1,\tau_1'} 1_Q$ are contained in $\Xi_{R_k^{-1/2}}$, while the distance between $\tau_1$ and $\tau_1'$ is $> R_1^{-1/2}$. Since $R_1^{-1/2} \geq (R_k^{-1/2})^{1/2}$ and \eqref{eq:Rk1/4} holds, the hypothesis of Lemma~\ref{lem_bilinear} is satisfied with $\kappa = R_1^{-1/2}$ and $\delta = R_k^{-1/2}$. From \eqref{bilrest}, we then obtain
\[
\int_Q |f_{k+1,\tau_1} f_{k+1,\tau_1'}|^2 \leq \frac{(R_k^{-1/2})^2}{R_1^{-1/2}} \int_Q |f_{k+1,\tau_1}|^2 \int_Q |f_{k+1,\tau_1'}|^2 = \frac{q^{r/2}}{|Q|} \int_Q |f_{k+1,\tau_1}|^2 \int_Q |f_{k+1,\tau_1'}|^2.
\]
The same inequality holds for $k=1$, because then $|f_{k+1,\tau_1}|$ and $|f_{k+1,\tau_1'}|$ are constants on squares of side length $R_1^{1/2}$.
Thus in either case, \eqref{eq:pruned2} is controlled by
\begin{align*}
\sum_{Q \in P_{R_{k}^{1/2}}(\Om_k)}\sum_{\tau_1 \neq \tau_{1}'}\int_{Q}|f_{k + 1, \tau_{1}}f_{k + 1, \tau_{1}'}|^{2}
&\leq q^{r/2}\sum_{Q \in P_{R_{k}^{1/2}}(\Om_k)}\frac{1}{|Q|} \sum_{\tau_{1} \ne \tau_{1}'} \int_{Q}|f_{k + 1, \tau_{1}}|^2 \int_{Q}|f_{k + 1, \tau_{1}'}|^2\\
&\leq q^{r/2}\sum_{Q \in P_{R_{k}^{1/2}}(\Om_k)}\frac{1}{|Q|}(\sum_{\tau_{1}}\int_{Q}|f_{k + 1, \tau_{1}}|^{2})^{2}
\end{align*}
where here $P_{R_{k}^{1/2}}(\Om_k)$ denotes the partition of $\Om_k$ into squares of side length $R_{k}^{1/2}$.
Since $Q$ has side length $R_{k}^{1/2}$, Plancherel and the definition of $g_{k}$ then controls this by
\begin{align*}
q^{r/2}\sum_{Q \in P_{R_{k}^{1/2}}(\Om_k)}\frac{1}{|Q|}(\sum_{\tau_k}\int_{Q}|f_{k + 1, \tau_k}|^{2})^{2} = q^{r/2}\sum_{Q \in P_{R_{k}^{1/2}}(\Om_k)}\frac{1}{|Q|}(\int_{Q}g_{k})^{2} = q^{r/2}\int_{\Om_k}g_{k}^{2}
\end{align*}
where the last equality is because $g_k$ is constant on squares of size $R_{k}^{1/2}$.

Therefore we have shown that
\begin{align*}
\alpha^{4}|U_{\alpha}(f) \cap \Om_k| \lsm q^{r/2} \int_{\Om_k}g_{k}^{2}.
\end{align*}
Using that we are in $\Om_k$ and applying the high lemma, this is controlled by
\begin{align} \label{eq:7.55}
(\log R)^2 q^{r/2} \int_{\Om_k}|g_{k}^{h}|^{2} \leq (\log R)^2 q^r \sum_{\tau_k}\int_{\F^2}|f_{k + 1, \tau_k}|^{4}
\end{align}
Write $f_{k+1,\tau_k} = \sum_{\tau_{k+1} \subset \tau_k} f_{k+1,\tau_{k+1}}$.
Note that the sum has $R_k^{-1/2}/R_{k+1}^{-1/2}$ terms. Using H\"{o}lder's inequality, we further obtain that
\begin{align*}
\eqref{eq:7.55} &\le (\log R)^2 q^r  (\frac{R_k^{-1/2}}{R_{k+1}^{-1/2}})^{3}  \sum_{\tau_{k + 1}} \int_{\F^2} |f_{k+1,\tau_{k+1}}|^4\\
& = (\log R)^2 q^{5r/2} \sum_{\tau_{k + 1}}\int_{\F^2}|f_{k + 1, \tau_{k + 1}}|^{4}\\
& = (\log R)^2 q^{5r/2} \sum_{\tau_{k + 1}}\int_{\F^2}\sum_{\st{T_{k + 1} \in \T(\tau_{k + 1})\\\|1_{T_{k + 1}}f_{k + 2, \tau_{k + 1}}\|_{L^{\infty}(\F^2)}\leq \ld}}|1_{T_{k + 1}}f_{k + 2, \tau_{k + 1}}|^{4}
\end{align*}
where in the last equality we have used that each $x \in \F^2$ is contained in exactly one $T_{k + 1} \in \T(\tau_{k + 1})$. Here we have also used the convention that if $k = J - 1$, then $f_{k + 2}$ is just $f$.
Applying the definition of $f_{k + 1}$ shows that this is
\begin{align}\label{eq:lambda2appear}
\leq (\log R)^2 q^{5r/2} \ld^{2}\sum_{\tau_{k + 1}}\int_{\F^2}|f_{k + 2, \tau_{k + 1}}|^{2} = (\log R)^2 q^{5r/2} \ld^{2}\int_{\F^2}|f_{k + 2}|^{2} \leq (\log R)^2 q^{5r/2} \ld^{2}\int_{\F^2}|f|^{2}
\end{align}
where the last inequality is by \eqref{eq:increase}.
Using \eqref{eq:lambda} then shows that we have proved
\begin{align*}
\alpha^{4}|U_{\alpha}(f) \cap \Om_k| \lsm (\log R)^6 q^{7r/2} \alpha^{-2} \nms{g_{J}}^{2}_{L^{\infty}(\F^2)} \int_{\F^2}|f|^{2}.
\end{align*}
It follows that
\begin{align} \label{eq:7.5conclusion}
\alpha^{6}|U_{\alpha}(f) \cap \Om_k| \lsm (\log R)^6 q^{7r/2} (\sum_{\tau_{J}}\|f_{\tau_{J}}\|_{L^{\infty}(\F^2)}^{2})^{2}\sum_{\tau_{J}}\|f_{\tau_{J}}\|_{L^{2}(\F^2)}^{2}.
\end{align}

\subsection{Estimating $\alpha^6 |U_{\alpha}(f) \cap L|$}

The right hand side of \eqref{eq:prunedlow} is
\begin{align}\label{eq:lowdomeq1}
&\leq \int_{L}(\sum_{\tau_1}|f_{1, \tau_1}|^{2})^{3}\leq \int_{L}(\sum_{\tau_1}|f_{2, \tau_1}|^{2})^{3}= \int_{L}g_{1}^{2}\sum_{\tau_{1}}|f_{2, \tau_{1}}|^{2}
\end{align}
where the second inequality is by \eqref{eq:prunedmon}.
For $x \in L$ and $k = 1, \dots, J-1$, we have
\begin{equation*}
g_k(x) \leq (1-(\log R)^{-1})^{-1} g_{k+1}(x)
\end{equation*}
so
\begin{equation*}
g_1(x) \lesssim \sum_{\tau_{J}} |f_{\tau_{J}}(x)|^2.
\end{equation*}
Therefore this and \eqref{eq:increase} shows that \eqref{eq:lowdomeq1} is
\begin{align*}
\lsm (\sum_{\tau_{J}}\nms{f_{\tau_{J}}}_{L^{\infty}(\F^2)}^{2})^{2}\int_{\F^2}|f|^{2}.
\end{align*}
It follows that
\begin{align} \label{eq:7.6conclusion}
\alpha^{6}|U_{\alpha}(f) \cap L| \lsm (\sum_{\tau_{J}}\|f_{\tau_{J}}\|_{L^{\infty}(\F^2)}^{2})^{2}\sum_{\tau_{J}}\|f_{\tau_{J}}\|_{L^{2}(\F^2)}^{2}.
\end{align}

Finally, we may sum \eqref{eq:7.5conclusion} over $k=1,\dots,J-1$ with \eqref{eq:7.6conclusion}. Since $J \leq N \lesssim \log R$, this concludes the proof of Proposition~\ref{prop:newprop34}, modulo the proof of Lemma~\ref{lem_bilinear}.

\subsection{Proof of Lemma~\ref{lem_bilinear}}
Decompose
\begin{equation*}
f_i = \sum_{\substack{\theta_i \subset I_i \\ |\theta_i| = \delta^{1/2}}} f_{i,\theta_i}.
\end{equation*}
Then by Plancherel,
\begin{equation*}
\int_{\F^2} |f_1 f_2|^2
= \sum_{\theta_1, \theta_1', \theta_2, \theta_2'} \int_{\F^2} f_{1,\theta_1} f_{2,\theta_2} \cdot \overline{ f_{1,\theta_1'} f_{2,\theta_2'} }
= \sum_{\theta_1, \theta_1', \theta_2, \theta_2'} \int_{\F^2} (\widehat{f_{1,\theta_1}}*\widehat{f_{2,\theta_2}}) \cdot  \overline{(\widehat{f_{1,\theta_1'}}*\widehat{f_{2,\theta_2'}})}.
\end{equation*}
For the last integral to be non-zero, the support of $\widehat{f_{1,\theta_1}}*\widehat{f_{2,\theta_2}}$ must intersect the support of $\widehat{f_{1,\theta_1'}}*\widehat{f_{2,\theta_2'}}$.
Thus we can find $(\xi_i, \eta_i)$, $i = 1, 2, 3, 4$ such that
$\xi_1 + \xi_2 = \xi_3 + \xi_4$ and $\eta_1 + \eta_2 = \eta_3 + \eta_4$
where $|\eta_{i} - \xi_i^{2}| \leq \delta$ and $\xi_1 \in \ta_1$, $\xi_2 \in \ta_2$,
$\xi_3 \in \ta_{1}'$, and $\xi_4 \in \ta_{2}'$.
Hence by the ultrametric inequality, for this $(\xi_1, \ldots, \xi_4)$, we have
\begin{equation} \label{linear}
\xi_1 + \xi_2 - (\xi_3 + \xi_4) = 0
\end{equation}
\begin{equation} \label{quad}
|\xi_1^2 + \xi_2^2 - (\xi_3^2 + \xi_4^2)| \leq \delta.
\end{equation}
From \eqref{linear}, we have $\xi_1 - \xi_4 = -(\xi_2-\xi_3)$, so we see from \eqref{quad} that
\begin{equation*}
|\xi_1-\xi_4| |\xi_1+\xi_4-(\xi_2+\xi_3)| \leq \delta.
\end{equation*}
But \eqref{linear} also implies $\xi_1+\xi_4-(\xi_2+\xi_3) = 2 (\xi_1 - \xi_3)$. Since $q$ is an odd prime, we have
\begin{equation*}
|\xi_1-\xi_4| |\xi_1 - \xi_3| \leq \delta.
\end{equation*}
Since $|\xi_1 - \xi_4| \geq \kappa$, this shows
\begin{equation*}
|\xi_1 - \xi_3| \leq \frac{\delta}{\kappa}.
\end{equation*}
If $\delta / \kappa \leq \delta^{1/2}$, i.e. \eqref{assump} holds,
then $|\xi_1 - \xi_3| \leq \delta^{1/2}$.
Since $\ta_1$ and $\ta_{1}'$ are intervals of length $\delta^{1/2}$
and two $q$-adic intervals of the same length are either disjoint or equal, we must have $\theta_1 = \theta_1'$. Using \eqref{linear} again then implies $\theta_2 = \theta_2'$.

This shows
\begin{equation*}
\int_{\F^2} |f_1 f_2|^2 = \sum_{\theta_1, \theta_2} \int_{\F^2} |\widehat{f_{1,\theta_1}}*\widehat{f_{2,\theta_2}}|^2 = \sum_{\theta_1, \theta_2} \int_{\F^2} |f_{1,\theta_1}|^2 |f_{2,\theta_2}|^2.
\end{equation*}
Now for $i = 1,2$, we may expand
\begin{equation*}
|f_{i,\theta_i}|^2 = \sum_{T_i \in \T(\theta_i)} |c_{T_i}|^2 1_{T_i}
\end{equation*}
as in Corollary \ref{cor:wavepacket}, so that $\sum_{T_i \in \T(\theta_i)} |c_{T_i}|^2 |T_i| = \int_{\F^2} |f_{i,\theta_i}|^2$.
Thus
\begin{equation*}
\begin{split}
\int_{\F^2} |f_{1,\theta_1}|^2 |f_{2,\theta_2}|^2
&= \int_{\F^2} \sum_{T_1 \in \T(\theta_1)} |c_{T_1}|^2 1_{T_1} \sum_{T_2 \in \T(\theta_2)} |c_{T_2}|^2 1_{T_2} \\
&= \sum_{T_1 \in \T(\theta_1)} \sum_{T_2 \in \T(\theta_2)} |c_{T_1}|^2 |c_{T_2}|^2 |T_1 \cap T_2|
\end{split}
\end{equation*}
Using the definition of $\kappa$, and Lemma \ref{lem:tubeintersect}, we see that
\begin{equation*}
|T_1 \cap T_2| \leq \delta^{-1/2} \cdot \frac{\delta^{-1/2}}{\kappa} = \frac{\delta^2}{\kappa} |T_1| |T_2| \quad \text{for all $T_1 \in \T(\theta_1)$, $T_2 \in \T(\theta_2)$},
\end{equation*}
so
\begin{equation*}
\int_{\F^2} |f_{1,\theta_1}|^2 |f_{2,\theta_2}|^2 \leq \frac{\delta^2}{\kappa} \int_{\F^2} |f_{1,\theta_1}|^2 \int_{\F^2} |f_{2,\theta_2}|^2.
\end{equation*}
Summing over $\theta_1$ and $\theta_2$ on both sides, we yield
\begin{equation*}
\int_{\F^2} |f_1 f_2|^2 \leq \frac{\delta^2}{\kappa} \int_{\F^2} |f_1|^2 \int_{\F^2} |f_2|^2,
\end{equation*}
as desired.

\bibliographystyle{amsplain}
\bibliography{log-bibliography}
\end{document}